\documentclass[11pt,reqno]{amsart}

\pdfoutput=1

\usepackage[verbose=true]{microtype}
\usepackage{amsmath}
\usepackage{amssymb}
\usepackage{amsthm}
\usepackage{bbm}
\usepackage{mathtools}
\usepackage{color}
\usepackage[colorlinks,linkcolor=black,linktocpage,citecolor=black,hypertexnames=true,pdfpagelabels]{hyperref}
\usepackage[capitalise]{cleveref}
\usepackage[all]{xy}
\usepackage{subcaption}
\usepackage{array}
\usepackage{enumerate}
\usepackage{tikz}
\usepackage{relsize}
\usepackage{diagbox}
\usepackage{tikz-cd} 
\usepackage{mdframed}
 \usepackage[foot]{amsaddr}
\usepackage{dsfont}

\newtheorem{theorem}{Theorem}

\newtheorem{corollary}[theorem]{Corollary}
\newtheorem{proposition}[theorem]{Proposition}

\newtheorem{definition}[theorem]{Definition}

\theoremstyle{definition}

\newtheorem{example}[theorem]{Example}
\newtheorem{remark}[theorem]{Remark}
\newtheorem{construction}[theorem]{Construction}

\renewcommand{\epsilon}{\varepsilon}

\newcommand{\N}{\mathbb{N}}
\newcommand{\C}{\mathbb{C}}

\newcommand{\be}{\begin{eqnarray}}
\newcommand{\ee}{\end{eqnarray}}
\newcommand{\ben}{\begin{enumerate}}
\newcommand{\een}{\end{enumerate}}

\newcommand{\ba}{\begin{array}}
\newcommand{\ea}{\end{array}}

\newcommand{\mc}{\mathcal}

\newcommand{\ogd}{$(\Omega,G)$-decompo\-sition} 
 
\newcommand{\omr}{{\rm rank}_{\Omega}}
\newcommand{\omgr}{{\rm rank}_{(\Omega,G)}}

\makeatletter
\def\p@subsection{}
\def\p@subsubsection{}
\makeatother

\usepackage[textsize=footnotesize]{todonotes}

\let\originalleft\left
\let\originalright\right
\renewcommand{\left}{\mathopen{}\mathclose\bgroup\originalleft}
\renewcommand{\right}{\aftergroup\egroup\originalright}

\newcommand\xqed[1]{%
  \leavevmode\unskip\penalty9999 \hbox{}\nobreak\hfill
  \quad\hbox{#1}}
\newcommand\demo{\xqed{$\triangle$}}

\begin{document}

\title{Tensor decompositions on simplicial complexes with invariance}

\author{Gemma De las Cuevas}
\address{Institute for Theoretical Physics, Technikerstr.\ 21a,  A-6020 Innsbruck, Austria}

\author{Matt Hoogsteder Riera}
\address{Institute for Theoretical Physics, Technikerstr.\ 21a,  A-6020 Innsbruck, Austria}

\author{Tim Netzer}
\address{Department of Mathematics, Technikerstr.\ 13,  A-6020 Innsbruck, Austria}

\date{\today}
\begin{abstract}
We develop a framework to analyse invariant decompositions of elements of tensor product spaces. Namely, we define an invariant decomposition with indices arranged on a simplicial complex, and which is explicitly invariant under a group action. We prove that this decomposition exists for all invariant tensors after possibly enriching the simplicial complex. As a special case we recover tensor networks with translational invariance and the symmetric tensor decomposition. We also define an invariant separable decomposition and purification form, and prove similar existence results. Associated to every decomposition there is a rank, and we prove several inequalities between them. For example, we show by how much the rank increases when imposing invariance in the decomposition, and that the tensor rank is the largest of all ranks. Finally, we apply our framework to nonnegative tensors, where we define a nonnegative and a positive semidefinite decomposition on arbitrary simplicial complexes with group action. We show a correspondence to the previous ranks, and as a very special case recover the nonnegative, the positive semidefinite, the completely positive and the completely positive semidefinite transposed decomposition.
\end{abstract}

\maketitle
\tableofcontents

\section{Introduction}

Tensor products appear prominently in almost all areas of mathematics, theoretical physics  and numerous other branches of science. Multilinear maps, higher-order derivatives and  homogeneous  polynomials can be seen as tensors, for example. In quantum mechanics, the state space of a multi-particle quantum system is modelled as a tensor product of the individual state spaces. Tensors, that is, elements of tensor product spaces, are also used in electrical engineering, psychometrics, data analysis (see \cite{Br10b} and references therein), and in relation to machine learning (see \cite{Gl19} and references therein), to cite a few examples.

It is a basic fact that every  element of a tensor product space $$v\in \mc{V}= \mc{V}_0\otimes\cdots \otimes \mc{V}_n$$ can be expressed as a finite sum of elementary tensors  $$v^{[0]}\otimes  \cdots \otimes v^{[n]}$$ where  $v^{[i]}\in \mc{V}_i$ for  $i=0,\ldots, n$.
A \emph{tensor network}, often used in quantum information theory and condensed matter physics \cite{Or18},
is a certain way of arranging the indices in this sum: 
they are chosen to reflect the physical arrangement of the individual degrees of freedom in a given quantum system.
For example, the indices could be arranged in a one-dimensional circle as in $$v = \sum_{\alpha_0,\ldots,\alpha_{n} =1}^r v_{\alpha_0,\alpha_1}^{[0]} \otimes v_{\alpha_1,\alpha_2}^{[1]} \otimes\cdots \otimes  v_{\alpha_{n-1},\alpha_{n}}^{[n-1]}\otimes  v_{\alpha_{n},\alpha_0}^{[n]}.$$Alternatively,  there could be a single joint index, 
\be\label{eq:tensdecomp}
v = \sum_{\alpha=1}^r v_\alpha^{[0]}\otimes v_\alpha^{[1]} \otimes \cdots \otimes v_\alpha^{[n]}.
\ee 
In either case, the smallest possible such $r$ defines the corresponding rank of the element $v$.

Symmetries play a  central role in theoretical physics and mathematics. 
Characterising the symmetries of a system is both of fundamental importance, as they reveal the conserved quantities, and of practical importance, as symmetric systems have fewer degrees of freedom, and thus allow for more efficient parametrisations. 
In this paper we focus on  \emph{external} symmetries. That is, 
we assume that a group $G$ acts on the set $\{0,1,\ldots, n\}$, and we consider the induced linear action of $G$ on $\mc V$, i.e.\ 
$$
g: v^{[0]} \otimes  \cdots \otimes v^{[n]} \mapsto v^{[g0]} \otimes  \cdots \otimes v^{[gn]}.
$$ 
We then say that $v\in \mc{V}$ has an external symmetry (given by $G$) if it is fixed by this action. 
In contrast, in \emph{internal} symmetries one typically has a representation $U_g: \mc{V}_i \to  \mc{V}_i$ of a group $G$, and one says that $v$ has an internal global symmetry if $(U_g)^{\otimes n} v=v$, or an internal local symmetry if $(U_g)^{\otimes l} v=v$ for a certain subset $l$ of subsystems, as in a lattice gauge theory. Internal symmetries  have been characterised in the context of tensor networks, e.g., in \cite{Sc10b,Si10b}.

In the context of the tensor decompositions considered above, if $v$ has an external symmetry, it is desirable to find a decomposition that makes this symmetry explicit, that is, to find an \emph{invariant} decomposition.  
For example, if the system is arranged in a one-dimensional circle with cyclic symmetry, an invariant  decomposition would be of the form 
$$ 
v = \sum_{\alpha_0,\ldots, \alpha_n=1}^r v_{\alpha_0,\alpha_1} \otimes v_{\alpha_1,\alpha_2} \otimes \cdots \otimes v_{\alpha_n,\alpha_0}. 
$$
This is known as the translationally invariant matrix product operator form, and the minimal such $r$ is called the  t.\ i.\ operator Schmidt rank \cite{De19}. 
Another example is a single joint index with full  symmetry, in which an invariant decomposition would be of the form
$$
v =  \sum_{\alpha=1}^r v_{\alpha} \otimes v_{\alpha} \otimes \cdots \otimes v_{\alpha},
$$
 which is known as a symmetric tensor decomposition, and the minimal such $r$ is called the symmetric tensor rank \cite{Co08c}.

In this paper we develop a theoretical framework to study  invariant tensor decompositions and their corresponding ranks, which specialises in particular to the above ones. 
Namely, we consider elements of tensor product spaces and express them as a sum of elementary tensor factors.
The indices in the sum are arranged over a simplicial complex $\Omega$, which is a well-studied object in topology.  
In addition, we consider a group $G$ acting on the simplicial complex. 
We then define a corresponding invariant tensor decomposition, called the \ogd, 
and an associated rank as the minimal number of terms of that decomposition, called $\omgr$. 
We then address the following questions: 
\emph{Does every invariant element have an \ogd?}
 Or, more precisely, 
 \emph{what are the conditions on $\Omega$ and $G$ that guarantee that every invariant element has an \ogd?}

 Our main result is  that such an invariant decomposition always exists,  provided that the indices are chosen and grouped in the right way  (\cref{thm:ogd}).
 More precisely, we show that every invariant element has an \ogd{}, 
 but only after possibly  raising the weights of the facets of the simplicial complex, and refining the group action.
This is the reason why we have to work with \emph{weighted simplicial complexes}, instead of just simplicial complexes.

In addition, we define the separable \ogd{} and the $(\Omega,G)$-purification form as a generalisation of the 
(translationally invariant) separable decomposition and (t.i.) purification form studied in \cite{De19}, and prove similar existence results. 
These decompositions incorporate different notions of  positivity into the local vectors.

 We also study inequalities between the ranks. 
 First, we  prove several inequalities between the rank, the separable rank and the purification rank on arbitrary simplicial complexes and with arbitrary group actions. 
Second, we study how these ranks are modified when changing the group action. 
For example, we study how much the rank increases when transforming a non-invariant decomposition into an invariant one.  
Third, we study how the ranks change when the simplicial complex is modified. 

Finally, we apply our framework to entry-wise nonnegative tensors.  
First we define the nonnegative and the positive semidefinite decomposition on arbitrary simplicial complexes with group action.  
These specialise to the nonnegative, positive semidefinite, completely positive and completely positive semidefinite transposed decompositions when the simplicial complex is an edge, and the group is trivial or the cyclic group of order 2. 
We then prove a correspondence with the previous decompositions  (\cref{thm:corresp}), thereby generalising the results of  \cite{De19}, 
and use it to prove inequalities for these new ranks.

This paper is organized as follows.
In \cref{sec:wsc} we  introduce the relevant notions related to simplicial complexes and group actions. 
In \cref{sec:inv} we  define the \ogd{} and prove our main existence results, among which \cref{thm:ogd} is the most important one.
In \cref{sec:ineq} we  prove inequalities between the ranks, and
in \cref{sec:nn} we apply our framework to several decompositions of nonnegative tensors.
We close with the conclusions an outlook in \cref{sec:outlook}.

\section{Weighted simplicial complexes and group actions}
\label{sec:wsc}

In this section we introduce the relevant notions of weighted simplicial complexes with group actions, which provides the underlying topological structure on which we will consider tensor decompositions. 
Specifically, in \cref{ssec:wsc} we define weighted simplicial complexes, and in 
\cref{ssec:actions}  group actions.
 We write $[n]$ for the set $\{0,\ldots, n\}$ and $\mathcal P_n$ for its power set $\mathcal P([n])$ throughout this paper.

\subsection{Weighted simplicial complexes}
\label{ssec:wsc}

We start by defining weighted simplicial complexes (see, for example, \cite{Da90} for more information). Examples are provided in \cref{ssec:actions} below. 

\begin{definition}\label{def:wsc}
(i) A \emph{weighted simplicial complex (wsc) on $[n]$}   is a function $$\Omega\colon \mathcal P_n\to \mathbb N$$ such that $S_1\subseteq S_2$ implies that   $\Omega(S_1)$ divides $\Omega(S_2).$  A wsc is called a \emph{simplicial complex (sc)} if it only takes values $0$ and $1$.
 
(ii) A set $S\in\mathcal P_n$ with $\Omega(S)\neq 0$ is called a \emph{simplex} of $\Omega$. We will assume throughout that each singleton $\{i\}$ is a simplex, and call the elements $i\in[n]$ the \emph{vertices} of the complex.
A maximal simplex (with respect to inclusion) is called a \emph{facet} of $\Omega.$
We denote by $$\mathcal F:=\left\{ F\in\mathcal P_n\mid  F\mbox{ facet of } \Omega \right\}$$ the set of all facets,  and  for each vertex $i\in[n]$ by  
$$\mathcal F_i:=\left\{F\in\mathcal F\mid i\in F\right\} $$the set of facets  that $i$ is contained in. 

The restriction of $\Omega$ to $\mathcal F$ and $\mathcal F_i$ makes these sets multisets, for which we use the notation $$\widetilde{\mathcal F} \ \mbox{ and  } \ \widetilde{\mathcal F}_i.$$  Each facet $F$ is contained in $\widetilde{\mathcal{F}}$ precisely $\Omega(F)$-many times.  There is the canonical \emph{collapse map} $$c\colon \widetilde{\mathcal F}\to \mathcal F, \quad c\colon \widetilde{\mathcal F}_i\to \mathcal F_i,$$ mapping all copies of a facet to the underlying facet.

(iii) Two vertices $i,j$ are \emph{neighbors}, if $$\mathcal F_i\cap\mathcal F_j\neq\emptyset \qquad (\mbox{equivalently if } \widetilde{\mathcal F}_i\cap\widetilde{\mathcal F}_j\neq\emptyset)$$ Two vertices are \emph{connected} if one can be reached from the other through a sequence of neighborly points. The wsc is \emph{connected} if any two vertices are connected.
 \demo\end{definition}

\begin{remark}\label{rem:wsc}
(i) A sc $\Omega$ is the characteristic function of a subset $\mathcal A\subseteq \mathcal P_n$. By the definition of a sc,  $\mathcal A$ is closed under passing to subsets. This is precisely how an (abstract) simplicial complex is usually defined. 

(ii) A wsc is a special case of a multihypergraph \cite{Bo98}, in which all simplices of a facet are contained, and  where the multiplicities satisfy  condition (i) of \cref{def:wsc}. 
Intuitively, one can think of a wsc as a well-formed mutihypergraph. 

For example, a multigraph without self-loops is a wsc in which every vertex has  value $1$ and every edge has the value given by its multiplicity in the multigraph. This applies in particular to every simple graph. 
Another example is a hypergraph \cite{Bo98} in which all simplices of a facet are contained: 
this is a sc in which the corresponding simplices have weight 1.

In fact, our framework could be formulated with multihypergraphs  (see also \cref{rem:sym}), but the slightly less general notion of a wsc is easier to define, digest and work with, in our opinion. 
\demo\end{remark}

\subsection{Group actions}\label{ssec:actions}

We start with some general definitions concerning group actions, and then consider actions on weighted simplicial complexes. We assume that the reader is familiar with the usual concept of a group acting on a set, as defined in \cite{La02} for example. The identity element of a group $G$ will always be denoted by $e$.

\begin{definition}\label{def:action}
(i) Let $G$ be a group acting on the sets $X$ and $Y$. A function $f\colon X\to Y$  is called  \emph{$G$-linear} if $$f(gx)=gf(x)$$ holds for all $x\in X,g\in G$. In case that $G$ acts trivially on $Y$, we instead call $f$ \emph{$G$-invariant}. 

(ii) If $G$ acts on $X$, then for any map $f\colon X\to Y$ and any $g\in G$ we define a new map \begin{align*}{}^gf\colon X&\to Y \\ x &\mapsto f(g^{-1}x). \end{align*}  
We have $${}^h\left({}^gf\right)={}^{hg} f\ \mbox{ and }\ {}^e f=f.$$ In particular, the mapping $f\mapsto{}^gf$ is a bijection on the set of all functions from $X$ to $Y$.
If $f$ is only defined on a subset $X'\subseteq X$, then ${}^gf$ is defined on $$gX'=\{ gx\mid x\in X'\}\subseteq X.$$

(iii) An action of $G$ on $X$ is \emph{free}, if  ${\rm Stab}(x)=\{e\}$ for every $x\in X$, where $${\rm Stab}(x):=\left\{ g\in G\mid gx=x\right\}.$$

(iv) An action of $G$ on $[n]$ is \emph{blending}, if whenever $\{g_00,\ldots,g_nn\}=[n]$ for certain $g_0,\ldots, g_n\in G$, then there is some $g\in G$ with $gi=g_ii$ for all  $i=0,\ldots, n$. 
\demo\end{definition}

We now introduce the main notion of a group action on a wsc.

\begin{definition}
(i) A \emph{group action of $G$ on the wsc $\Omega$} consists of the following:
 \begin{itemize}
 \item An action of $G$ on $[n]$, such that $\Omega$ is $G$-invariant with respect to the induced action of $G$ on $\mathcal P_n$ (i.e.\ the action permutes vertices in such a way that simplices become simplices of the same weight). This then induces an action of $G$ on $\mathcal F$. 
 \item An action of $G$ on $\widetilde{\mathcal F}$, such that the collapse map $$c\colon{\widetilde{\mathcal F}}\to\mathcal F$$ is $G$-linear (we also say the action of $G$ on $\widetilde{\mathcal F}$ \emph{refines} the action of $G$ on $\mathcal F$). 
  \end{itemize}

(ii) An action of $G$ on the wsc $\Omega$ is called \emph{free} if the action of $G$ on $\widetilde{\mathcal F}$ is free.
\demo\end{definition}

\begin{remark}\label{rem:grac}(i) Since a wsc has finitely many vertices, we will usually assume that the group $G$ is finite as well. Since groups always act by permutations, we could also assume that $G$ is a subgroup of the permutation group $S_{n+1}$, but sometimes it is much more convenient not to choose the latter representation.

(ii) A group action permutes the vertices $[n]$ of the wsc $\Omega$ in a way that preserves the weighted adjacency structure. 
The action induces an action of $G$ on $\mathcal F$, where facets from the same orbit have the same weight. 
Each $g\in G$  provides a weight-preserving bijection \begin{align*}g\colon \mathcal F_i&\to\mathcal F_{gi}\\ F&\mapsto gF.\end{align*}

(iii) To obtain a group action on a wsc, one has to provide additional information, namely how elements $g\in G$ permute the different copies of facets when passing from a facet $F$ to the facet $gF$. 
Clearly, a group action can always be refined (that is, defined compatibly on $\widetilde{\mathcal F}$), but there might be more than one way to do so. 
For each vertex $i$ and each group element $g$, we then obtain the following commutative diagram:
$$
 \xymatrix{\widetilde{\mathcal F}_i\ar[r]^g\ar[d]_c & \widetilde{\mathcal F}_{gi} \ar[d]^c \\ \mathcal F_i \ar[r]_{F\mapsto gF} & \mathcal F_{gi}}
$$

(iv) For a sc, the notion of group action precisely covers the usual notion of a group acting by automorphisms. 

(v) The notion of a blending group action just refers to the action of $G$ on $[n]$. It means that any possible permutation within any orbit is provided by some group element.

(vi) The notion of a free group action  on a wsc involves the action of $G$ on $\widetilde{\mathcal F}$. Note that an action of $G$ on $\Omega$ can be free, without the underlying  action of $G$ on $[n]$ or on $\mathcal F$ being free. 
In fact, any action of $G$ on $\Omega$ can be refined to a free action, after possibly increasing the weights of the facets, as we will show in \cref{prop:extend}. In combination with \cref{thm:ogd}, this is why we consider weighted simplicial complexes instead of just simplicial complexes here.

(vii) An action $G$ on a set $X$ is free if and only if there is a $G$-linear map $${\bf z}\colon X\to G$$ where $G$ acts on itself by left-multiplication (this  action is clearly free). To define ${\bf z}$ for a free action, choose an element $x$ in each orbit and map $gx$ to $g$. The other implication is clear.
\demo\end{remark}

\begin{example}\label{ex:wsc}
(i) The sc $\Sigma_n$ that maps each subset of $[n]$ to $1$  is called the \emph{$n$-simplex}. 
\bigskip
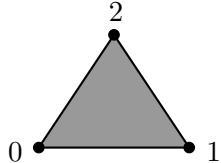
\begin{center}
\begin{tikzpicture}
\filldraw (-1,0) circle (2pt);
\filldraw (1,0) circle (2pt);
\filldraw (0,1.5) circle (2pt);
\draw[thick] (-1,0) -- (1,0) -- (0,1.5)-- cycle;
\filldraw[opacity=0.4] (-1,0) -- (1,0) -- (0,1.5);
\put (-40,-5) {$0$};
\put (35,-5) {$1$};
\put (-2,48) {$2$};
\end{tikzpicture}
\end{center}

\bigskip\noindent
It has only one (multi-)facet, i.e.\ $\mathcal F=\widetilde{\mathcal F}=\{[n]\}$.  Any group action on $[n]$ is a group action on $\Sigma_n$. The action of the full permutation group on $[n]$ is blending. 
The only free action on $\Sigma_n$  is the  action from the trivial group.   However, if the weight of the (only) facet is raised to $\vert G\vert$, any action from $G$ on $[n]$ has a free refinement, as we will see in \cref{prop:extend}.

(ii) For $n\geq 1$, the \emph{complete graph} $\mathcal K_n$ is the sc with weight one on all sets $\{i,j\}$, for $i,j\in [n]$, and otherwise $0$. 

\bigskip
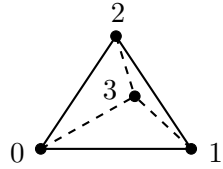
\begin{center}
\begin{tikzpicture}
\filldraw (-1,0) circle (2pt);
\filldraw (1,0) circle (2pt);
\filldraw (0,1.5) circle (2pt);
\filldraw (0.25,0.7) circle (2pt);
\draw[thick] (-1,0) -- (1,0) -- (0,1.5)-- cycle;
\draw[thick,dashed] (1,0) -- (0.25,0.7);
\draw[thick, dashed] (0,1.5) -- (0.25,0.7);
\draw[thick, dashed] (-1,0) -- (0.25,0.7);
\put (-40,-5) {$0$};
\put (35,-5) {$1$};
\put (-2,48) {$2$};
\put (-5,19) {$3$};
\end{tikzpicture}
\end{center}

This sc has $n+ 1 \choose 2$ facets. Again, any group action on $[n]$ is a group action on $\mathcal K_n$. The action of the full permutation group is blending but not free.

(iii) For $n\geq 1$, the \emph{line of length $n$} is the sc $\Lambda_n$ corresponding to  the following graph:

\bigskip
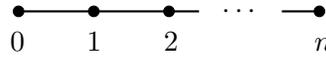
\begin{center}
\begin{tikzpicture}
\filldraw (-2,0) circle (2pt);\put (-60,-15){$0$};
\filldraw (-1,0) circle (2pt);\put (-31,-15){$1$};
\filldraw (0,0) circle (2pt);\put (-2,-15){$2$};
\put (20,-2.5) {$\cdots$};
\filldraw (2,0) circle (2pt);\put (55,-15){$n$};
\draw[thick] (-2,0) -- (-1,0);
\draw[thick] (-1,0) -- (0,0);
\draw[thick] (0,0) -- (0.4,0);
\draw[thick] (1.5,0) -- (2,0);
\end{tikzpicture}
\end{center}

\bigskip\noindent The set $\mathcal F=\widetilde{\mathcal F}$ has $n$ elements.
The only non-trivial group action on $\Lambda_n$ is by the cyclic group with two elements $G=C_2$, where the generator inverts the order of vertices, i.e.\ vertex $i$ is sent to $n-i$. 
This action is free if and only if $n$ is even, and blending if and only if  $n \leq 2$. For $n$ odd, the action admits a free refinement if the weight of the middle edge is increased to $2$.

(iv) For $n\geq 3$, the \emph{circle of length $n$} is the sc $\Theta_n$  corresponding to  the following graph:

\bigskip
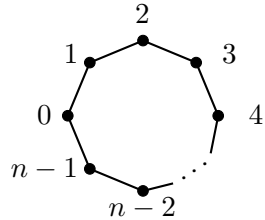
\begin{center}
\begin{tikzpicture}
\filldraw (-1,0) circle (2pt);\put (-40,-3) {$0$};
\filldraw (-0.707,0.707) circle (2pt);\put (-30,20) {$1$};
\filldraw (0,1) circle (2pt);\put (-3,35) {$2$};
\filldraw (0.707,0.707) circle (2pt);\put (30,20) {$3$};
\filldraw (1,0) circle (2pt);\put (40,-3) {$4$};
\filldraw (0,-1) circle (2pt);\put (-13,-38) {$n-2$};
\filldraw (-0.707,-0.707) circle (2pt);\put (-50,-23) {$n-1$};
\draw[thick](-1,0)--(-0.707,0.707) -- (0,1)--(0.707,0.707) -- (1,0)-- (0.9,-0.5);
\draw[thick](0.4,-0.9)--  (0,-1)-- (-0.707,-0.707)--(-1,0) ;
\put(13,-23) {\rotatebox[origin=c]{35}{$\cdots$}};
\end{tikzpicture}
\end{center}

\bigskip\noindent It has $n$ facets.  There is for example the canonical action of the cyclic group $G=C_n$. 
It is free but not blending.

(v) We have $\Sigma_1=\mathcal K_1=\Lambda_1$, which is just the simple edge, having precisely one (multi)-facet. The only interesting group action is by $C_2=S_2$, which is blending but not free (although the action on $\{0,1\}$ is free!).
The \emph{double edge}  is the wsc $\Delta$ on $\mathcal P_1$ that assigns the value $1$ to $\{0\}, \{1\}$ and the value $2$ to $\{0,1\}.$

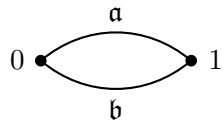
\bigskip
\begin{center}
\begin{tikzpicture}
\filldraw (-1,0) circle (2pt);\put (-40,-3){$0$};
\filldraw (1,0) circle (2pt);\put (35,-3){$1$};
\draw[thick] (-1,0) to[out=-40, in=220] (1,0); \put (-3,15){$\mathfrak a$};
\draw[thick] (-1,0) to[out=40, in=140] (1,0);\put (-3,-22){$\mathfrak b$};
\end{tikzpicture}
\end{center}

\bigskip\noindent
In this case 
  $$\mathcal F_0=\mathcal F_1=\mathcal F=\{ \{0,1\}\}$$ are singletons, but  $$\widetilde{\mathcal F}_0=\widetilde{\mathcal F}_1=\widetilde{\mathcal F}=\{\mathfrak a,\mathfrak b\}$$ are not.   If $C_2$ also flips $\mathfrak a$ and $\mathfrak b$, the action is free.
    
 (vi) Let $G\neq\{e\}$ be a nontrivial finite group with generating set $e\notin S\subseteq G$. We first define the Cayley graph $$\Gamma(G,S)=(V,E)$$ as the oriented graph with vertex set $V=G$, where $(g,h)\in E$ is an edge if and only if there exists some $s\in S$ with $gs=h$. We now define the corresponding wsc $\mathcal C(G,S)$, called the Cayley complex, on the vertex set $V$ by assigning the value $1$ to all vertices, and for $g\neq h\in V$ $$\mathcal C(G,S)\left(\{g,h\}\right):=\left\{\begin{array}{cl} 2 &\colon \{(g,h),(h,g)\}\subseteq  E \\ 1 &\colon  \#\left(\{(g,h),(h,g)\}\cap E\right)= 1 \\ 0 & \colon \mbox{else.} \end{array} \right.$$ Since $S$ is a generating set for $G$, the wsc $\mathcal C(G,S)$ is connected, and its facets are all of cardinality $2$. Their weights indicate whether there are one or two oriented edges between the corresponding vertices in the Cayley graph.  
 The set $\widetilde{\mathcal F}_g$ is  identified with $S\times\{{\rm in},{\rm out}\}$, for each $g\in V$, 
and the set $\widetilde{\mathcal F}$ is  identified with $E$. 
 
 The group $G$ acts on itself by left-multiplication. This provides a free action of $G$ on the wsc $\mathcal C(G,S),$ by letting $G$ act on the elements of $\widetilde{\mathcal F}=E$ entry-wise.  Note that the double edge $\Delta$ and the circle $\Theta_n$ are special cases of this construction, where $G=C_2$ and $C_n$, respectively, and 
 the generating set is $S=\{1\}$. 
\demo\end{example}

We  now prove what we have already seen in \cref{ex:wsc} (i), (iii) and (v). 

\begin{proposition}\label{prop:extend}
Any action of the finite group $G$ on the wsc $\Omega$ has a free refinement  after possibly increasing the weights of the facets of $\Omega$.
\end{proposition}
\begin{proof} 
Assume $G=\left\{ g_1,\ldots, g_r\right\}$ with  $r=\vert G\vert$.  Let  $\bar\Omega$  be the wsc obtained by multiplying  the weights of all facets  of $\Omega$ by $r$.
Assume $\Omega(F)=m$ for some $F\in\mathcal F$. We denote the $m$ copies of $F$ in $\widetilde{\mathcal F}$ by $F_1,\ldots, F_m$. For any $g\in G$ we know that $gF_1,\ldots, gF_m$ are the copies of $gF$. Now label the $rm$ copies of $F$ in the new multiset $\bar{\mathcal F}$ by $$F_1^{g_1},\ldots, F_1^{g_r}, \ldots, F_m^{g_1},\ldots, F_m^{g_r},$$ i.e.\ every copy $F_i\in\widetilde{\mathcal F}$ is replaced by   $r$ duplicates, indexed by the group elements. The collapse map $\bar c\colon\bar{\mathcal F}\to \mathcal F$ factors through $\widetilde{\mathcal F}$ via the partial collapse map $$\tilde c\colon\bar{\mathcal F}\to \widetilde{\mathcal F};\ F_i^{g}\mapsto F_i.$$
$$\xymatrix{\bar{\mathcal F} \ar[dr]_{\bar c} \ar[r]^{\tilde c} & \widetilde{\mathcal F}\ar[d]^{c} \\ & \mathcal F}$$
  We now define $$g\cdot F_i^{h}:= (gF_i)^{gh}$$ as the $gh$-th duplicate of the facet $gF_i$. This defines an action of $G$ on $\bar{\mathcal F}$ which makes $\tilde c$ a $G$-linear map, i.e.\ the action refines the given action on $\widetilde{\mathcal F}$. Since the action of $G$ on itself by left-multiplication is  free, this new action on $\bar\Omega$ is free. \end{proof}

Remark that we do not claim that this refinement is optimal. 
In the above construction, the weight of every facet is multiplied by the cardinality of $G$, but 
there might another free refinement that increases less the multiplicity of each facet.


\section{Invariant tensor decompositions and ranks}
\label{sec:inv}

In this section we define and study several different tensor decompositions and tensor ranks on a wsc. 
Specifically, in \cref{ssec:inv} we prove the main result of this paper, namely the existence of invariant decompositions in many cases. 
In \cref{ssec:sep} we specialise to the separable decomposition,
and in \cref{ssec:puri} to the purification form. 
We also prove subadditivity of the ranks and provide many examples along the way.

Throughout this section, we fix a wsc $\Omega$ with an  action from the group $G$.  
For each $i\in[n]$ we fix  a  $\mathbb C$-vector space     $\mathcal V_{i}$ (called the \emph{local vector space at site $i$}), and whenever $i,j$ are in the same orbit, the spaces must coincide. Unless otherwise mentioned,  we do not impose any further  conditions on the spaces $\mathcal V_i$---they can be, for example, infinite dimensional.
We  define the \emph{global vector space} as $$\mathcal V:=\mathcal V_{0}\otimes\cdots\otimes\mathcal V_{n}.$$ Note that we always consider the algebraic tensor product, even if the spaces are infinite dimensional and have more structure, such as Hilbert spaces. So by the very definition, every element of $\mathcal V$ is a finite sum of elementary tensors.

The action of $G$ on $[n]$  induces a linear action on $\mathcal V$, by permuting the tensor factors. An element $v\in\mathcal V$ is called \emph{$G$-invariant} if it is invariant under this action. The subspace of invariant elements is denoted  $\mathcal V_{\rm inv}$.

For two sets $X$ and $Y,$ the set $Y^X$ contains, by definition, all functions from $X$ to $Y$. If $X$ is finite, such a function is often written as the tuple of its values. For example,  for any set $\mathcal I$ a function  $\alpha\in\mathcal I^{\widetilde{\mathcal F}}$ could be written as a tuple with entries from $\mathcal I$, indexed over the finite set $\widetilde{\mathcal F}$.  However, the functional point of view turns out to be much more flexible and less technical in most of our proofs below. For example, when $\alpha$ is seen as a function,  for  $i\in[n]$ the  restriction of $\alpha$ to $\widetilde{\mathcal F}_i$ $$\alpha_{\mid_{ \widetilde{\mathcal F}_i}}\in\mathcal I^{\widetilde{\mathcal F}_i}$$ is easily defined in the obvious way. We also call it the \emph{restriction of $\alpha$ to $i$} and  write $\alpha_{\mid_i}$ instead. In tuple notation this means that we erase those entries of the tuple whose indices do not belong to $\widetilde{\mathcal F}_i$, thus making it possibly shorter. We will stick to the functional viewpoint whenever possible, and use the tuple notation only in some explicit example.

\subsection{The invariant decomposition}
\label{ssec:inv}

 We now define the basic invariant tensor decomposition, called the $(\Omega,G)$-decomposition. Explicit examples of such decompositions are provided in \cref{ex:ogd} below. Since the definition might look quite abstract at the beginning, let us first explain the main idea behind it.  We  consider  sums of elementary tensors, where these indices take finitely many values. The summation indices will be placed on the facets of $\Omega$. Each elementary tensor in such a  sum is composed of vectors from the local spaces, and each local vector is indexed by just those facets  that contain the corresponding vertex. The reader might want to have a look at \cref{ex:ogd} first to get an idea on how the decompositions look like.

\begin{definition}\label{def:ogd}
(i) For $v\in\mathcal V$, an \emph{$(\Omega,G)$-decomposition of $v$} consists of  a  finite set $\mathcal I$ and  families $$V^{[i]}=\left(v^{[i]}_{\beta}\right)_{\beta\in \mathcal I^{\widetilde{\mathcal F}_i}}$$  with all $v^{[i]}_{\beta}\in \mathcal V_i$,  for all $i\in[n]$,   such that: 
\begin{itemize} 
\item[(a)] We have $$v=\sum_{\alpha\in\mathcal I^{\widetilde{\mathcal F}}}v^{[0]}_{\alpha_{\mid 0}}\otimes v^{[1]}_{\alpha_{\mid 1}}\otimes\cdots\otimes v^{[n]}_{\alpha_{\mid n}}.$$
\item[(b)] For all $i\in[n], g\in G$ and   $\beta\in \mathcal I^{\widetilde{\mathcal F}_i} $ we have $$v_{\beta}^{[i]}=v_{{}^g \beta}^{[gi]}$$ where ${}^g\beta$ is the function defined in \cref{def:action} (ii). 
\end{itemize}

(ii) The smallest cardinality of an index set $\mathcal I$  among all $(\Omega,G)$-decompo\-sitions of $v$  is called the \emph{$(\Omega,G)$-rank of $v$}, denoted  $${\rm rank}_{(\Omega,G)}(v).$$  

(iii) For the trivial group action  we call an \ogd{} just \emph{$\Omega$-decomposition}, and write $\omr(v)$ for the rank.
\demo\end{definition}

\begin{remark}
(i) The family $V^{[i]}$ from \cref{def:ogd} (i) is called the \emph{local tensor at site $i$}, and its elements are called the \emph{local vectors at site $i$}. 
Condition (a) specifies how the local vectors need to be combined in order to obtain $v$. 
The wsc $\Omega$ hereby determines which indices to use at the different vertices, and thus  how different sites interact  locally. Condition (b) takes into account 
invariance with respect to the group action, by specifying how local vectors along  orbits must coincide. 
The $(\Omega,G)$-rank expresses the minimal number of local vectors needed to express $v$. 

 (ii) We adopt the convention to set ${\rm rank}_{(\Omega,G)}(v)=\infty$ if $v$ does not admit an $(\Omega,G)$-decomposition.
\demo\end{remark}

\begin{example}\label{ex:ogd}
(i)  Consider the  $n$-simplex $\Sigma_n$, for $n\geq 1$. Since 
$$\widetilde{\mathcal{F}}_i=\widetilde{\mathcal F}=\{\{0,1, \ldots, n\}\}$$ 
is a singleton for each $i\in[n]$, a $\Sigma_n$-decomposition only uses one index and is thus of the form $$v=\sum_{\alpha=1}^r v_\alpha^{[0]}\otimes\cdots\otimes v_\alpha^{[n]}.$$  Thus ${\rm rank}_{\Sigma_n}(v)$ is the smallest number of elementary tensors needed to obtain $v$ as their sum, which is also known as the \emph{tensor rank} of $v$.

Now assume that the action of a group $G$ on $[n]$ is transitive, i.e.\ there is only one orbit. Then condition (b) requires that the  $(\Sigma_n,G)$-decompo\-sition  of $v$ uses the same local vectors everywhere. It is thus of the form $$\sum_{\alpha=1}^r v_\alpha\otimes\cdots\otimes v_\alpha,$$ which is also known as a \emph{symmetric tensor decomposition}, and   ${\rm rank}_{(\Sigma_n,G)}(v)$ is also called the \emph{symmetric tensor rank} of $v$ (see for example \cite{Co08c}). We will come back to this point in \cref{rem:sym}.

(ii) Consider the complete graph $\mathcal K_3$. A $\mathcal K_3$-decomposition of $v$ is of the form  $$v=\sum_{i,j,k,l,m,n=1}^r v_{ijk}^{[0]}\otimes v_{ilm}^{[1]}\otimes v_{jln}^{[2]}\otimes v_{kmn}^{[3]}.$$ If $S_{[3]}$ 
is the full permutation group acting on $\{0,1,2,3\}$, a $(\mathcal K_3,S_{[3]})$-decomposition of $v$ is $$ v=\sum_{i,j,k,l,m,n=1}^r v_{ijk}\otimes v_{ilm}\otimes v_{jln}\otimes v_{kmn},$$ with the additional property that the local tensor $$\left(v_{ijk}\right)_{i,j,k=1}^r$$ is fully symmetric.  
Note that this decomposition reveals the same symmetry as (i), namely  invariance under the full symmetry group $S_{[n]}$, but 
the simplicial complex is different. 
This illustrates how in an $(\Omega,G)$-decomposition both elements are important: the simplicial complex  and the group action. 

(iii) For $n\geq 1$ consider $\Lambda_n$, the  line of length $n$, where a $\Lambda_n$-decomposition of $v$ has the form $$v=\sum_{\alpha_0,\ldots, \alpha_{n-1}=1}^r v_{\alpha_0}^{[0]}\otimes v_{\alpha_0,\alpha_1}^{[1]}\otimes\cdots\otimes v_{\alpha_{n-2},\alpha_{n-1}}^{[n-1]}\otimes v_{\alpha_{n-1}}^{[n]}.$$ This is also called a \emph{matrix product operator form} of $v$, and  the $\Lambda_n$-rank is also called the \emph{operator Schmidt rank}---see for example \cite{De19} and references therein. 

For $n=2$, with action of $C_2$ (acting as $0\mapsto 2$,  $1\mapsto 1$,  $2\mapsto 0$),
 a $(\Lambda_2,C_2)$-decomposition is  $$v=\sum_{\alpha,\beta=1}^r v_\alpha\otimes w_{\alpha,\beta}\otimes v_\beta$$ with the additional property that $w_{\alpha,\beta}=w_{\beta,\alpha}$ for all $\alpha,\beta$ (i.e.\ the local tensor at site $1$ is symmetric).

(iv) For $n\geq 3$ consider the  circle $\Theta_n$ of length $n$. A $\Theta_n$-decompo\-sition has the form $$v=\sum_{\alpha_0,\ldots, \alpha_{n-1}=1}^r v_{\alpha_{0},\alpha_1}^{[0]}\otimes v_{\alpha_1,\alpha_2}^{[1]}\otimes\cdots\otimes v_{\alpha_{n-1},\alpha_{0}}^{[n-1]}.$$ This is almost the same as the decomposition from (iii), but with closed (i.e.\ periodic) boundary conditions.
Now let the cyclic group $C_n$ act on $\Theta_n$. A $(\Theta_n,C_n)$-decomposition then is  $$v=\sum_{\alpha_0,\ldots, \alpha_{n-1}=1}^r v_{\alpha_{0},\alpha_1}\otimes v_{\alpha_1,\alpha_2}\otimes\cdots\otimes v_{\alpha_{n-1},\alpha_{0}},$$ which is called \emph{translational invariant matrix product operator form} in \cite{De19}, and the corresponding rank is called the \emph{t.i.\ operator Schmidt rank}.

(v) On the simple edge $\Sigma_1=\mathcal K_1=\Lambda_1$, the decomposition is $$v=\sum_{\alpha=1}^r v_\alpha^{[0]}\otimes v_\alpha^{[1]}$$ and the $C_2$-invariant decomposition  is  $$v=\sum_{\alpha=1}^r v_\alpha\otimes v_\alpha.$$ On the double edge $\Delta$,  with free action as in \cref{ex:wsc} (v), the corresponding decompositions are $$v=\sum_{\alpha,\beta=1}^r v_{\alpha,\beta}^{[0]}\otimes v_{\beta,\alpha}^{[1]}$$ and   $$v=\sum_{\alpha,\beta=1}^r v_{\alpha,\beta}\otimes v_{\beta,\alpha}.$$ The difference between the simple edge and the double edge has been observed and examined in \cite{De19}, but without developing the theoretical  foundations, as we do here.
\demo\end{example}

Our first result on the existence of decompositions does not involve a group action yet:

\begin{theorem}\label{thm:nogroup}  For  every connected wsc $\Omega$  and every $v\in\mathcal V$, we have ${\rm rank}_\Omega(v)<\infty.$  
\end{theorem}
\begin{proof}We start with a decomposition $$v=\sum_{j\in\mathcal I} w_j^{[0]}\otimes \cdots \otimes w_j^{[n]}$$ of $v$ as a finite sum of elementary tensors. 
For $i\in[n]$ and  $\beta\in{\mathcal I}^{\widetilde{\mathcal F}_i}$ we then define 
$$v^{[i]}_{\beta}:= \left\{ \begin{array}{ll}w_{j}^{[i]}&\colon\beta \mbox{ takes the constant value }  j\in\mathcal I \\  0&\colon\mbox{else.}\end{array}\right.$$
Since $\Omega$ is connected, for $\alpha\in\mathcal I^{\widetilde{\mathcal F}}$ the functions $\alpha_{\mid_i}$ are  constant only if $\alpha$ is constant. This implies $$\sum_{\alpha\in\mathcal I^{\widetilde{\mathcal F}}} v_{\alpha_{\mid_0}}^{[0]}\otimes\cdots\otimes v_{\alpha_{\mid_n}}^{[n]}=\sum_{j\in\mathcal I} w_j^{[0]}\otimes\cdots\otimes w_j^{[n]}=v,$$ which proves the claim.
\end{proof}

\begin{remark}\label{rem:sym}
Clearly not every $v\in\mathcal V$ 
admits an $(\Omega,G)$-decomposition, since such a decomposition for example requires $G$-invariance: 
\begin{align*}g\cdot v&=\sum_{\alpha\in\mathcal I^{\widetilde{\mathcal F}}}v^{[g0]}_{\alpha_{\mid g0}}\otimes \cdots\otimes v^{[gn]}_{\alpha_{\mid gn}} \\
&=\sum_{\alpha\in\mathcal I^{\widetilde{\mathcal F}}}v^{[g0]}_{{}^g((^{g^{-1}}\alpha)_{\mid 0})}\otimes \cdots\otimes v^{[gn]}_{{}^g((^{g^{-1}}\alpha)_{\mid n})}
\\
&=\sum_{\alpha\in\mathcal I^{\widetilde{\mathcal F}}}v^{[0]}_{({}^{g^{-1}}\alpha)_{\mid 0}}\otimes \cdots\otimes v^{[n]}_{({}^{g^{-1}}\alpha)_{\mid n}}\\ &=\sum_{\alpha\in\mathcal I^{\widetilde{\mathcal F}}}v^{[0]}_{\alpha_{\mid 0}}\otimes \cdots\otimes v^{[n]}_{\alpha_{\mid n}} =v.
\end{align*} For the third equation we have used condition (b) from \cref{def:ogd}, and for the fourth that ${}^{g^{-1}}\alpha$ runs through $\mathcal I^{\widetilde{\mathcal F}}$ if $\alpha$ does. 

However, an $(\Omega,G)$-decomposition might imply an even stronger symmetry than  $G$-invariance of $v$. In \cref{ex:ogd} (i) we have seen that all transitive group actions on the $n$-simplex lead to the same $(\Sigma_n,G)$-decomposition, which is the fully symmetric decomposition. So if, for example, $v$ is only invariant under the cyclic group $C_n$, it cannot have  a $(\Sigma_n,C_n)$-decomposition. One way around this problem would be to use the multiset $\widetilde{\mathcal S}$ of \emph{all} simplices instead of $\widetilde{\mathcal F}$ in all definitions, since the action of $G$ on $\widetilde{\mathcal S}$ determines the action on $[n]$. However, this would not allow us to cover the fully symmetric tensor decomposition from \cref{ex:ogd} (i)  anymore. 

Another way around this problem, which is the one we have chosen here, is to raise the weights of the facets so that the action becomes free (\cref{prop:extend}). 
Freeness suffices to prove that every invariant element has an invariant decomposition, as we will see in \cref{thm:ogd}.
\demo\end{remark}

The following  is our main result on the existence of invariant tensor  decompositions. In combination with \cref{prop:extend} it shows that a $G$-invariant decomposition exists for every invariant tensor, after possibly enriching the underlying topological structure (or, in a less complicated formulation, by using more indices at each site).

\begin{theorem}[Main result]\label{thm:ogd}
 Let the action of $G$ on the connected wsc $\Omega$ be free. Then for every $v\in\mathcal V_{\rm inv}$ we have $\omgr(v)<\infty.$
Moreover,  for  every decomposition  $$v=\sum_{j} w_j^{[0]}\otimes\cdots\otimes w_j^{[n]}$$ of $v$ as a sum of elementary tensors, there is  an $(\Omega,G)$-decomposition of $v,$  which uses only nonnegative multiples of the $w_j^{[i]}$ as its local vectors.
\end{theorem}

Note the theorem does not say anything about the value of $\omgr(v)$. 
It says that, provided the action of $G$ on $\Omega$ is free, 
an \ogd{}  exists, and the proof provides a (generally non-optimal) way to obtain it. 

\begin{proof} Using  \cref{rem:grac} (vii), we fix a $G$-linear map ${\bf z}\colon \widetilde{\mathcal F} \to G$. For $v\in\mathcal V_{\rm inv}$ we first choose  an $\Omega$-decomposition, whose existence we have proven in \cref{thm:nogroup} (we can choose the local vectors from any initial tensor decomposition, as it  is clear from the proof). The local tensors from the $\Omega$-decomposition  are denoted $$W^{[i]}=\left( w_\beta^{[i]}\right)_{\beta\in\mathcal I^{\widetilde{\mathcal F}_i}}$$ for $i\in[n]$. We define the new index set $$\widetilde{\mathcal I}:= \mathcal I \times G$$ and consider the projection maps $$p_1\colon\widetilde{\mathcal I}\to \mathcal I,\quad p_2\colon\widetilde{\mathcal I}\to G.$$ For each $i\in[n]$ and   $\beta\in \widetilde{\mathcal I}^{\widetilde{\mathcal F}_i}$  we now  define 
  $$v^{[i]}_\beta:= \left\{ \begin{array}{ll} w_{{}^g(p_1\circ \beta)}^{[gi]} &:  p_2\circ\beta = ({}^{g^{-1}}\mathbf z)_{\mid_i}   
  \\ 0 &:\mbox{else.}\end{array}\right.$$ Note that if such a $g$ exists for $p_2\circ\beta$, it is uniquely determined, since ${\bf z}$ is $G$-linear and the action of $G$ on itself by left-multiplication is free. Also note that the new  local vectors are all among the inital vectors. 
  
  The arising local tensors now fulfill (b) from \cref{def:ogd} (i), as one easily checks.
  We now compute
 \begin{align*}
  &\sum_{\tilde\alpha\in\widetilde{\mathcal I}^{\widetilde{\mathcal F}}} v_{\tilde\alpha_{\mid_0}}^{[0]} \otimes\cdots\otimes  v_{\tilde\alpha_{\mid_n}}^{[n]}\\
=& \sum_{\small\begin{array}{c}z\in G^{\widetilde{\mathcal F}}\\ \forall i  \exists g_i\colon z_{\mid_i}=({}^{g_i^{-1}}\mathbf z)_{\mid_i} \end{array}}  \sum_{\alpha\in\mathcal I^{\widetilde{\mathcal F}}}  w_{{}^{g_0}(\alpha_{\mid_0})}^{\left[g_00\right]} \otimes\cdots\otimes w_{{}^{g_n}(\alpha_{\mid_n})}^{\left[g_nn\right]}.
 \end{align*}
 Since $\Omega$ is connected, $\mathbf z$ is $G$-linear,  and the action of $G$ on itself is free, we immediately obtain $g_i=g_j=:g$ for all $i,j$, if $z$ fulfills the above conditions. 
  So for each fixed $z$, the sum simplifies to $$\sum_{\alpha\in\mathcal I^{\widetilde{\mathcal F}}}w_{\alpha_{\mid_{g0}}}^{\left[g0\right]} \otimes\cdots\otimes w_{\alpha_{\mid_{gn}}}^{\left[gn\right]}$$ for some $g$ depending on $z$.  But this is just $g\cdot v$, and since $v$ is $G$-invariant it is in fact $v$. So the total sum yields a positive multiple of $v$ (the sum is not empty, since at least $z=\mathbf z$ fulfills the conditions). Since a positive scaling factor can be absorbed into the local vectors, this proves the claim. 
\end{proof}

\begin{remark}(i) 
Freeness of the action of $G$ on $\Omega$  in \cref{thm:ogd} is necessary to obtain an $(\Omega,G)$-decomposition for every invariant vector. 
For example, if a group action on the  $n$-simplex is transitive on $[n]$, an $(\Omega,G)$-decomposition requires full symmetry, which is stronger than $G$-invariance in general. However, such an action is never free.

(ii) Even if an $(\Omega,G)$-decomposition exists for all invariant vectors, in general the local vectors cannot be chosen from  any  initial  tensor decomposition. This also requires freeness of the action, that is, freeness is also necessary for the second statement of  \cref{thm:ogd}. 
One example is the simple edge with action from $C_2$, 
for which an \ogd{} exists for each invariant vector (by \cref{thm:genex} below), 
but one cannot choose the local vectors from any initial tensor decomposition. This is only possible on the double edge, where the action is indeed free. This will be studied in \cref{ssec:sep}. 
\demo\end{remark}

\begin{example}(i) The cyclic action of $C_n$ on the circle $\Theta_n$ is free, so every invariant vector admits a $(\Theta_n,C_n)$-decomposition, or, in the words of \cite{De19}, a translational invariant matrix product operator form.

(ii) More generally, whenever $G$ is a finite group with generating set $S$ as in \cref{ex:wsc} (vi), the action of $G$ on  $\mathcal C(G,S)$ is free, and the invariant decomposition thus exists for every invariant vector.
\demo\end{example}

We will prove another existence result for decompositions below, for which we need the following basic inequalities:
\begin{proposition}\label{prop:subadd}
Let $G$ act on the connected wsc $\Omega$. Then for all $v,w\in\mathcal V$  the following is true:

(i)  $\omgr(v+w)\leq\omgr(v)+\omgr(w).$

(ii) If all $\mathcal V_i$ are algebras, then  $\omgr(vw)\leq\omgr(v)\: \omgr(w).$
\end{proposition}
\begin{proof} Both statements are clearly true if either $v$ or $w$ does not admit an $(\Omega,G)$-decomposition.
So let $$ V^{[i]}=\left(v_{\beta}^{[i]}\right)_{\beta\in\mathcal I^{\widetilde{\mathcal F}_i}},\quad W^{[i]}=\left(w_{\beta}^{[i]}\right)_{\beta\in\mathcal J^{\widetilde{\mathcal F}_i}}$$ be the  local tensors from $(\Omega,G)$-decompositions of $v$ and $w$. 

For (i) we take the direct sum of the local tensors to obtain an \ogd{} for $v+w$. In detail, we define the new index set $\mathcal L:=\mathcal I \sqcup \mathcal J$ as the disjoint union of $\mathcal I$ and $\mathcal J$, set $$x_\beta^{[i]}:=\left\{ \begin{array}{ll} v_{\beta}^{[i]}&\colon\beta \mbox{ takes values only in } \mathcal I\\   w_{\beta}^{[i]}&\colon\beta \mbox{ takes values only in }\mathcal J\\ 0 &\colon\mbox{else} \end{array}\right.$$ for $\beta\in\mathcal L^{\widetilde{\mathcal F}_i}$, and obtain new local tensors $$X^{[i]}:=V^{[i]}\oplus W^{[i]}:=\left( x_\beta^{[i]} \right)_{\beta\in\mathcal L^{\widetilde{\mathcal F}_i}}.$$ Then condition (b) from \cref{def:ogd} is clearly fulfilled, and using connectedness of $\Omega$ one immediately checks $$\sum_{\alpha\in\mathcal L^{\widetilde{\mathcal F}}}x^{[0]}_{\alpha_{\mid 0}}\otimes \cdots\otimes x^{[n]}_{\alpha_{\mid n}}=v+w.$$
Since $\vert\mathcal L\vert=\vert\mathcal I\vert +\vert\mathcal J\vert,$ the statement is proven.

For (ii) we take the tensor product of the local tensors. In detail, consider the new index set $\mathcal L:=\mathcal I\times\mathcal J$ with the two projections $p_1\colon\mathcal L\to\mathcal I, p_2\colon\mathcal L\to \mathcal J,$ and define 
$$x_\beta^{[i]}:= v_{p_1\circ\beta}^{[i]}w_{p_2\circ\beta}^{[i]}$$ for $\beta\in\mathcal L^{\widetilde{\mathcal F}_i}$. The new local tensors $$X^{[i]}:=V^{[i]}\otimes W^{[i]}:=\left(x^{[i]}_\beta\right)_{\beta\in\mathcal L^{\widetilde{\mathcal F}_i}}$$ then provide an \ogd{} for $vw$, with $\vert\mathcal L\vert=\vert\mathcal I\vert\cdot\vert\mathcal J\vert$.
\end{proof}

The following is our second result on the existence of invariant decompositions. It is a consequence of the symmetric decomposition of symmetric tensors in finite dimension \cite{Co08c}.
 
\begin{theorem}\label{thm:genex}
Let the action of $G$ on the connected wsc $\Omega$ be blending. Then for every $v\in\mathcal V_{\rm inv}$ we have $\omgr(v)<\infty.$
\end{theorem}
\begin{proof}
We start with a decomposition $$v=\sum_{j\in\mathcal I} w_j^{[0]}\otimes \cdots \otimes w_j^{[n]}$$ of $v$ as a finite sum of elementary tensors. We then choose complex numbers $d_{\ell}^{[i]}$ for $i=0,\ldots, n$ and $\ell=1,\ldots, r$ (for some large enough  $r$), such that the following holds: 
\be\label{eq:symtens}
\sum_{\ell=1}^r d^{[i_0]}_\ell\cdots d^{[i_n]}_\ell=\left\{\begin{array}{cl}  1 \colon&  \{ i_0,\ldots, i_n\}=[n]\\ 0 \colon & \mbox{else.} \end{array}\right. 
\ee
 This is in fact just a symmetric tensor decomposition of the symmetric tensor defined by the right hand side, so the existence of such numbers follows from \cite[Lemma 4.2]{Co08c}.
For $i\in[n], \ell=1,\ldots, r$ and  $\beta\in{\mathcal I}^{\widetilde{\mathcal F}_i}$ we now define 
$$v^{[i]}_{\ell,\beta}:= \left\{ \begin{array}{ll}\sum_{g\in G} d_\ell^{[gi]}w_{j}^{[gi]}&\colon\beta \mbox{ takes the constant value }  j\in\mathcal I \\  0&\colon\mbox{else.}\end{array}\right.$$
 For each fixed $\ell$, these vectors fulfill  condition (b) of \cref{def:ogd} (i) and thus provide an \ogd{} of a certain element $v_\ell\in\mathcal V$. We now compute \begin{align*}v_1+\cdots + v_r&=\sum_{\ell=1}^r \sum_{\alpha\in\mathcal I^{\widetilde{\mathcal F}}} v_{\ell,\alpha_{\mid_0}}^{[0]}\otimes\cdots\otimes  v_{\ell,\alpha_{\mid_n}}^{[n]} \\ &=\sum_{g_0,\ldots, g_n\in G} \sum_{\ell=1}^r d_\ell^{[g_00]}\cdots d_\ell^{[g_nn]}\sum_{j\in\mathcal I}w_{j}^{[g_00]}\otimes\cdots\otimes w_{j}^{[g_nn]}.\end{align*} For the last equation we have again used that $\Omega$ is connected, so if $\alpha_{\mid_i}$ is constant for all $i$, then $\alpha$ is constant. By the choice of the $d_\ell^{[i]}$ and  the fact that the group action is blending, this simplifies further (where $\sim$ stands for ``some positive multiple of"): 
 \begin{align*} v_1+\cdots +v_r&= \sum_{\tiny\begin{array}{cc}g_0,\ldots, g_n\in G \\ \{ g_00,\ldots, g_nn\}=[n]\end{array}} \sum_{j\in\mathcal I}w_{j}^{[g_00]}\otimes\cdots\otimes w_{j}^{[g_nn]}\\ &\sim\sum_{g\in G}\sum_{j\in\mathcal I}w_{j}^{[g0]}\otimes\cdots\otimes w_{j}^{[gn]}\\ &= \sum_{g\in G} g \cdot v\\ &\sim v.
 \end{align*} 
 Now   \cref{prop:subadd} (i) implies  $\omgr(v)\leq r\vert\mathcal I\vert<\infty.$
\end{proof}

\begin{example}
Any fully symmetric tensor admits a symmetric tensor decomposition $$v=\sum_{\alpha=1}^r v_\alpha\otimes\cdots\otimes v_\alpha.$$ This is the statement of \cref{thm:genex} for the full symmetric group acting on the  $n$-simplex $\Sigma_n$. This is indeed not very surprising, since we have used the result for finite-dimensional local spaces in our proof. But  \cref{thm:genex} shows that the result also holds for infinite-dimensional spaces.

Another decomposition for fully symmetric tensors comes from the complete graph $\mathcal K_n$, as  illustrated  in \cref{ex:ogd} (ii) for $n=3$. This decomposition also exists for invariant vectors, but is obviously weaker than the one on the simplex (it can be constructed in an obvious way from the simplex-decomposition).
\demo\end{example}

\subsection{The invariant separable decomposition}
\label{ssec:sep}

Separability and its negation, entanglement, are 
central notions in quantum information theory.
We will now formulate and study separable invariant tensor decompositions in our framework. Throughout this section we thus assume that each local space  $$\mathcal V_i=\mathcal B(\mathcal H_i)$$ is the space of bounded operators on some 
(not necessarily finite dimensional)
Hilbert space $\mathcal H_i$, and again consider the global space $$\mathcal V=\mathcal V_0\otimes \cdots\otimes \mathcal V_n=\mathcal B(\mathcal H_0)\otimes \cdots\otimes\mathcal B(\mathcal H_n).$$

\begin{definition}(i) 
An element $\sigma\in\mathcal V$ is \emph{separable} if it admits a decomposition  $$\sigma=\sum_{j=1}^r \sigma_j^{[0]}\otimes \cdots\otimes \sigma_j^{[n]}$$ where all $\sigma_j^{[i]}\in \mathcal B(\mathcal H_i)$ are positive semidefinite (psd) operators.

(ii) For $\sigma\in\mathcal V$, a \emph{separable \ogd} is an \ogd $$\sigma=\sum_{\alpha\in\mathcal I^{\widetilde{\mathcal F}}}\sigma^{[0]}_{\alpha_{\mid 0}}\otimes\cdots\otimes \sigma^{[n]}_{\alpha_{\mid n}}$$ in which all  local operators $\sigma_\beta^{[i]}$ are psd. 

(iii) The smallest cardinality of the index set $\mathcal I$ among all separable \ogd{} of $\sigma$ is called the \emph{separable $(\Omega,G)$-rank}, denoted  $${\rm sep\mbox{-}rank}_{(\Omega,G)}(\sigma).$$

(iv) In case of the trivial group action, we call a separable \ogd{}  just \emph{separable $\Omega$-decomposition}, and write ${\rm sep\mbox{-}rank}_{\Omega}(\sigma)$ for the separable rank.
\end{definition}

\begin{remark} (i) A separable \ogd{} of $\sigma$ can clearly  exist only if $\sigma$ is separable and $G$-invariant.

(ii) In the case of the line $\Lambda_n$, a separable $\Lambda_n$-decomposition is called the \emph{separable decomposition} in \cite{De19}. 
For a circle $\Theta_n$ with cyclic group action $C_n$, a separable $(\Theta_n,C_n)$-decomposition is called a \emph{translational invariant separable decomposition} in \cite{De19}.
\demo\end{remark}

We now easily obtain our main result on the existence of separable $(\Omega,G)$-decompositions.

\begin{theorem}\label{thm:sep}
Let  the action of $G$  on the connected wsc $\Omega$ be free. Then for every separable $\sigma\in\mathcal V_{\rm inv}$ we have ${\rm sep\mbox{-}rank}_{(\Omega,G)}(\sigma)<\infty.$
\end{theorem}
\begin{proof}
This is a direct consequence of \cref{thm:ogd}, when starting with a decomposition of $v$ as a sum of psd elementary tensors.
\end{proof}

\begin{example}
(i) For every wsc $\Omega,$ every separable $v\in\mathcal V$ has a separable $\Omega$-decompo\-sition. This is true since the action of the trivial group  is free.

(ii) On the circle $\Theta_n$ (for $n\geq 3$), every cyclically invariant and separable $\rho$  admits a separable $(\Theta_n,C_n)$-decomposition.

(iii) On the simple edge with action from $C_2$, not every separable $v\in \mathcal V_{\rm inv}$ admits a separable $(\Lambda_1,C_2$)-decomposition, which would be of the form $$\sigma=\sum_{\alpha=1}^r \sigma_\alpha\otimes\sigma_\alpha$$ with all $\sigma_\alpha$ psd. This was shown in \cite{De19},  using the fact that not every symmetric and entry-wise nonnegative matrix has a \emph{completely positive factorization} (this also follows from the results in \cref{sec:nn}). This shows that, even if an $(\Omega,G)$-decomposition exists for all invariant vectors, 
 the second statement of \cref{thm:ogd} (concerning how the local vectors can be chosen) needs freeness of the action. 
 On the double edge, the action is free and  the separable decomposition thus exists---namely, it is of the form \begin{equation}\sigma=\sum_{\alpha,\beta=1}^r \sigma_{\alpha,\beta}\otimes\sigma_{\beta,\alpha}.\tag*{}\end{equation}

(iv) More generally, the separable decomposition exists for each invariant separable tensor on the Cayley complex 
$\mathcal C(G,S)$ from \cref{ex:wsc} (vi).
\demo\end{example}

\begin{proposition}
Let $G$ act  on the connected wsc $\Omega$. Then for  $\rho,\sigma \in\mathcal V$ we have   $${\rm sep\mbox{-}rank}_{(\Omega,G)}(\rho+\sigma)\leq{\rm sep\mbox{-}rank}_{(\Omega,G)}(\rho)+{\rm sep\mbox{-}rank}_{(\Omega,G)}(\sigma).$$
\end{proposition}
\begin{proof}Just follow the proof of \cref{prop:subadd} (i).
\end{proof}

\begin{remark}\label{rem:cones}
Instead of a separable $(\Omega, G)$-decomposition, one could define decompositions where the local vectors must be taken from certain specified cones in each space $\mathcal V_i$. The separable decomposition just corresponds to the case of the cone of psd operators. \cref{thm:ogd} ensures that such invariant decompositions exist for all invariant elements that admit a standard tensor decomposition with such vectors, since scaling with positive reals is compatible with  cones. In \cref{sec:nn}, when defining the nonnegative decomposition for finite-dimensional tensors, we will use make use of this fact.
\demo\end{remark}

\subsection{The invariant purification form}
\label{ssec:puri}

The separable \ogd, which reveals the positivity of the element, exists only for separable elements.
The purification form is another decomposition that reveals the positivity of the element, and which exists for all positive elements. 
We will now introduce an $(\Omega, G)$-purification form and study it in our framework.
As in the last section we assume that each local space  $$\mathcal V_i=\mathcal B(\mathcal H_i)$$ is the space of bounded operators on some 
Hilbert space $\mathcal H_i$. We again consider the global space $$\mathcal V=\mathcal V_0\otimes \cdots\otimes \mathcal V_n=\mathcal B(\mathcal H_0)\otimes \cdots\otimes\mathcal B(\mathcal H_n)\subseteq \mathcal B(\mathcal H_0\otimes \cdots \otimes\mathcal H_n).$$ With the inclusion on the right we can define what it means for an element from $\mathcal V$ to be psd.
We will denote the set of bounded linear operators from  $\mathcal H_i$ to $\mathcal H'_i$ by $\mathcal B(\mathcal H_i,\mathcal H'_i)$.

\begin{definition}\label{def:puri} 
(i) For $\sigma\in\mathcal V$ an \emph{$(\Omega,G)$-purification} is an element 
$$\xi\in \mathcal B(\mathcal H_0,\mathcal H'_0)\otimes\cdots\otimes \mathcal B(\mathcal H_n,\mathcal H'_n)$$ 
with $$\sigma=\xi^*\xi \ \mbox{ and }\ \omgr(\xi)<\infty, $$ 
where ${}^*$ indicates the adjoint. 
Here, the $\mathcal H'_i$ can be arbitrary Hilbert spaces.

(ii) The smallest $(\Omega,G)$-rank  among all $(\Omega,G)$-purifications of $v$  is called the \emph{$(\Omega,G)$-purification rank} of $\sigma$, denoted $${\rm puri\mbox{-}rank}_{(\Omega,G)}(\sigma).$$

(iii) In case of the trivial group action, we again just  say \emph{$\Omega$-purification} and  \emph{$\Omega$-purification rank}, denoted ${\rm puri\mbox{-}rank}_{\Omega}(\sigma).$
\demo\end{definition}

\begin{remark} An   $(\Omega,G)$-purification of  $\sigma$ can clearly  exist only if $\sigma$ is positive semidefinite and $G$-invariant. Positive semidefiniteness is obvious, since Hermitain squares are always psd. Invariance follows from the fact that the set of invariant elements is a $*$-subalgebra of $\mathcal V$.
\demo\end{remark}

We now easily obtain a result on the existence of invariant purifications in many cases:

\begin{theorem}\label{prop:puriex} Let $G$ act on the wsc $\Omega$, and  assume  $\omgr(\xi)<\infty$ holds for every $\xi\in\mathcal V_{\rm inv}$. Further assume that all $\mathcal H_i$ are finite-dimensional.  Then for every positive semidefinite   $\sigma\in\mathcal V_{\rm inv}$ we have ${\rm puri\mbox{-}rank}_{(\Omega,G)}(\sigma)<\infty.$
\end{theorem}

Note that the assumption of the theorem is fulfilled if, for example, the action of $G$ is free (as in \cref{thm:ogd}) or blending (as in  \cref{thm:genex}).

\begin{proof}
If all $\mathcal H_i$ are finite dimensional, then $\mathcal V=\mathcal B(\mathcal H_0\otimes \cdots\otimes\mathcal H_n),$ and thus the positive semidefinite operator $\sigma\in\mathcal V$ admits a (unique) positive semidefinite square root $\xi\in\mathcal V$. This $\xi$  is in fact a polynomial expression in $\sigma$, and thus also $G$-invariant. 
 By assumption, $\xi$ admits an \ogd{} and is thus an $(\Omega,G)$-purification of $\sigma.$
\end{proof}

Note that the proof uses the square root of $\sigma$, which is a special case of a purification (see also \cite{De19}).

\begin{proposition}
Let $G$ act on the connected wsc $\Omega$. Then for $\rho,\sigma \in\mathcal V$ we have  $${\rm puri\mbox{-}rank}_{(\Omega,G)}(\rho+\sigma)\leq{\rm puri\mbox{-}rank}_{(\Omega,G)}(\rho)+{\rm puri\mbox{-}rank}_{(\Omega,G)}(\sigma).$$
\end{proposition}
\begin{proof} We can assume that both $\rho$ and $\sigma$ admit  $(\Omega,G)$-purifications 
 \begin{align*}\xi\in&\ \mathcal B(\mathcal H_0,\mathcal H'_0)\otimes\cdots\otimes \mathcal B(\mathcal H_n,\mathcal H'_n)\\\chi\in&\ \mathcal B(\mathcal H_0,\mathcal H''_0)\otimes\cdots\otimes \mathcal B(\mathcal H_n,\mathcal H''_n).\end{align*}
 We understand both $\xi$ and $\chi$ as elements from $$\mathcal B(\mathcal H_0,\mathcal H'_0\times\mathcal H''_0)\otimes\cdots\otimes \mathcal B(\mathcal H_n,\mathcal H'_n\times\mathcal H''_n),$$ by letting the local operators  from $(\Omega,G)$-decompositions map to the first/second components, respectively. This way we obtain $$(\xi+\chi)^*(\xi+\chi)=\xi^*\xi+\chi^*\chi=\rho+\sigma,$$ and to  $\xi+\chi$ we can apply \cref{prop:subadd} (i).
\end{proof}

\section{Some inequalities}
\label{sec:ineq}

In this section we derive some inequalities, 
first between the several ranks we have introduced  (\cref{ssec:ineqdiff}), 
then between different group actions on the same complex (\cref{ssec:chgroup}),
 and finally between different complexes (\cref{ssec:chcomplex}). 

\subsection{Inequalities between different ranks} 
\label{ssec:ineqdiff}

To compare the different ranks, we fix an action of $G$ on the connected  wsc $\Omega$. We further assume that $\mathcal V_i=\mathcal B(\mathcal H_i)$ for each $i\in[n]$.
The following statement is a  generalisation of the results in \cite{De19}:
\begin{proposition}\label{prop:ineq1}
For each $\sigma\in\mathcal V$ we have 
\begin{itemize}
\item[(i)] $\omgr(\sigma)\leq{\rm sep\mbox{-}rank}_{(\Omega,G)}(\sigma)$
\item[(ii)] ${\rm puri\mbox{-}rank}_{(\Omega,G)}(\sigma)\leq {\rm sep\mbox{-}rank}_{(\Omega,G)}(\sigma)$
\item[(iii)] $\omgr(\sigma)\leq {\rm puri\mbox{-}rank}_{(\Omega,G)}(\sigma)^2.$
\end{itemize}
\end{proposition}
\begin{proof}  Since a separable \ogd{} is a special case of an $(\Omega,G)$-decomposition, (i) is clear.  For (ii) assume a separable \ogd{} for $\sigma$ over the index set $\mathcal I$ exists. Denote the local psd operators by $\sigma_\beta^{[i]}\in\mathcal V_i$ and define $$\tau_\beta^{[i]}:= \sqrt{\sigma_\beta^{[i]}}\in\mathcal V_i.$$ Then  consider the bounded operator
 \begin{align*}
\xi_\beta^{[i]}\colon \mathcal H_i & \to \mathcal H'_i:=\mathcal H_i\times \cdots \times\mathcal H_i \\
h & \mapsto (0,\ldots, 0, \tau_\beta^{[i]}h,0,\ldots, 0)
 \end{align*}
where the product runs over $\mathcal I^{\widetilde{\mathcal F}_i}$, and $\tau_\beta^{[i]}h$ appears in the entry indexed by $\beta$. 
These local operators fulfill (b) from \cref{def:ogd} and thus provide an $(\Omega,G)$-decomposition of some $\xi\in\mathcal V$, with $\omgr(\xi)\leq \vert \mathcal I\vert$. By construction we have $$\left(\xi_\beta^{[i]}\right)^*\xi_\gamma^{[i]} =\delta_{\beta,\gamma}\cdot  \sigma_\beta^{[i]}$$
where $\delta_{\beta,\gamma}$ is the Kronecker delta. 
This immediately implies $\xi^*\xi=\sigma,$ so we obtain ${\rm puri\mbox{-}rank}_{(\Omega,G)}(\sigma)\leq\vert \mathcal I\vert$, the desired result.
For (iii) we let $\xi$ be an  $(\Omega,G)$-purification of $\sigma$ and compute $$ \omgr(\sigma)=\omgr(\xi^*\xi)\leq \omgr(\xi^*)\omgr(\xi)=\omgr(\xi)^2$$ where we use the construction from \cref{prop:subadd} (ii). 
\end{proof}

\begin{remark}
(i) The purification rank cannot be upper bounded by a function of the rank alone, in general. This happens already in the case $\Lambda_1$,  the simple edge,  without group action.
Similarly, the separable rank cannot be upper bounded by a function of the purification rank, and thus of the rank, in general. This is true on the simple edge again, with and without a group action. This was shown in \cite{De19, De13c, Go12}, using a connection to factorisations of nonnegative matrices,  which is generalised in \cref{thm:corresp} below. 

(ii) There exist upper  bounds on the different ranks in terms of the dimension of the global space $\mathcal V$. We  refer the reader to \cite{De19}; the generalisations to the more general framework from this work are straightforward. 
\demo\end{remark}

\subsection{Changing the group}
\label{ssec:chgroup}

We now compare the ranks with respect to different group actions. So  throughout this section we fix an action of $G$ on the connected  wsc $\Omega,$  and  let  $H\subseteq G$ be a subgroup. Then the restricted action is an action of $H$ on $\Omega$, and we can compare the ranks with respect to $G$ and $H$.  Since   an $(\Omega, G)$-decomposition is also an $(\Omega, H)$-decomposition,  we clearly have    $${\rm rank}_{(\Omega,H)}(v)\leq {\rm rank}_{(\Omega,G)}(v)$$ for  all $v\in\mathcal V,$ and the same is true for the separable rank and the purification rank. 
 The first nontrivial result is about free actions as in \cref{thm:ogd}:

\begin{proposition}\label{thm:ineqfree} Let the action of $G$ on $\Omega$ be free and let $H$ be a normal subgroup of $G$. Then  for every $G$-invariant $v\in\mathcal V$ we have $${\rm rank}_{(\Omega,G)}(v)\leq \vert G/H\vert\cdot  {\rm rank}_{(\Omega,H)}(v),$$ and in particular $${\rm rank}_{(\Omega,G)}(v)\leq \vert G\vert\cdot  {\rm rank}_{\Omega}(v).$$ 
 The same inequalities also  hold for the separable rank.
\end{proposition}
\begin{proof} The proof is almost the same as the one of \cref{thm:ogd}. This time start with an $(\Omega,H)$-decomposition of $v$, define $\widetilde{\mathcal I}=\mathcal I\times G/H$ and use the $G$-linear map $${\bf z'}\colon\widetilde{\mathcal F}\stackrel{\bf z}{\to} G\stackrel{\rm pr}{\to} G/H$$ 
instead of ${\bf z}$, where ${\rm pr}$ denotes the canonical projection map. 
All constructions are well-defined since $H$ is a normal subgroup, and result in an $(\Omega,G)$-decomposition of $v$. From $$\vert\widetilde{\mathcal I}\vert=\vert\mathcal I\times G/H\vert=\vert \mathcal I\vert\cdot\vert G/H\vert$$ the result follows.
\end{proof}

\begin{example}(i) For cyclically invariant elements on the circle $\Theta_n$ we obtain that the $(\Theta_n,C_n)$-rank is always at most $n$ times the $\Theta_n$-rank. This statement is given in \cite[Proposition 60]{De19}.

(ii) More generally, for a finite group $G$ with generating set $S$ as in \cref{ex:wsc} (vi), and for $G$-invariant elements on $\mathcal C(G,S)$, we obtain that 
the $G$-invariant rank is always at most $\vert G\vert$  times the rank. 
\demo\end{example}

The following result  concerns blending group actions and the trivial subgroup. For the proof we need a strengthening of the property of blending, called \emph{strong blending}. 
We say that the action of $G$ on $\Omega$ is strongly blending if whenever $\{g_00,\ldots,g_nn\}=[n]$   there exists some $g\in G$ such that  $gi=g_ii$ \emph{and} $g,g_i$ act identically on $\widetilde{\mathcal F}_i$, for all $i\in[n]$. On the  $n$-simplex,  for example, this is equivalent to  blending.

\begin{proposition}\label{prop:inreg}
There is a function $C\colon\N\to\N$ such that whenever the action of a group $G$ is strongly blending 
on the wsc $\Omega$,  we have $$ {\rm rank}_{(\Omega,G)}(v)\leq C(n)\cdot  {\rm rank}_{\Omega}(v)$$ for all $G$-invariant $v\in\mathcal V$. 
\end{proposition}

Recall that the $n$ on the right hand side of the equation refers to the vertex set of $\Omega$, which is $[n]$.

\begin{proof}The proof  is similar to the one of \cref{thm:genex}. One starts with an $\Omega$-decomposition of $v$, with local vectors $w_\beta^{[i]}$, and defines  $$v^{[i]}_{\ell,\beta}:= \sum_{g\in G} d_\ell^{[gi]}w_{ {}^g\beta}^{[gi]},$$ with the same numbers $d_\ell^{[i]}$ as in the proof of \cref{thm:genex}. The rest of the proof is similar. Note that  to replace $$w_{{}^{g_i}(\alpha_{\mid_i})}^{[g_ii]}\ \mbox{ by } \ w_{{}^{g}(\alpha_{\mid_i})}^{[gi]}$$ in the last step of the proof, we need  strong blending of the action (this is also why the proof of \cref{thm:genex} did not start with an $\Omega$-decomposition of $v$, but with a decomposition as a sum of elementary tensors).

Note that $C(n)$ is called $r$ in the proof of \cref{thm:genex}, and  is  simply the symmetric tensor rank of the tensor defined in the right hand side of \eqref{eq:symtens}. 
\end{proof}

\begin{remark}(i) For the $n$-simplex $\Sigma_n$ and the full permutation group $G$, the $\Sigma_n$-rank is the tensor rank and the $(\Sigma_n,G)$-rank is the symmetric tensor rank. Until recently is was unknown whether these two ranks always coincide for symmetric tensors (this was known as \emph{Comon's conjecture}). It was then  shown in \cite{Sh18} that the symmetric tensor rank can be strictly larger than the tensor rank.  It seems that not much is known about  bounds of the symmetric tensor rank in terms of the tensor rank in general.  \cref{prop:inreg} applies to this case and gives a bound on the symmetric tensor rank in terms of the tensor rank and  $n$. For example, in the case $n=2$ (i.e.\ three-partite tensors) this bound is $C(2)=3$, as one easily checks. 

We conjecture that $C(n)=n+1$ for all $n$. In view of \cite{Zh16} it would be enough to show that the tensor rank of tensor defined on the right hand side of \eqref{eq:symtens} is at most $n+1$. We have verified this with Mathematica for up to $n=9$.

(ii) It is unclear whether \cref{prop:inreg} also holds for the purification rank. On the other hand, it clearly does not hold for the separable rank, since a separable $(\Omega,G)$-decomposition might not even exist for a $G$-invariant element, even if a separable $\Omega$-decomposition exists.
\demo\end{remark}

\subsection{Changing the simplicial complex}
\label{ssec:chcomplex}

Throughout this section let $$\Omega,\Psi\colon \mathcal P_n\to \N$$ be two weighted simplicial complexes on $[n]$.  We want to compare the $\Omega$-rank to the $\Psi$-rank of elements $v\in\mathcal V$. For simplicity, we do not employ a group action here. The following general construction will be used in all of the below results. It is in essence a generalisation of the idea of the proof of \cref{thm:nogroup}.

\begin{construction}\label{const:change}(i) We denote by $\widetilde{\mathcal F}(\Omega)\ \mbox{and}\ \widetilde{\mathcal F}(\Psi)$ the multisets of facets of $\Omega$ and $\Psi$, respectively. 
Assume we are given $v\in\mathcal V$ and a $\Psi$-decomposition using local vectors $w_\beta^{[i]}\in\mathcal V_i$ for  $i\in[n]$ and  $$\beta\colon \widetilde{\mathcal F}(\Psi)_i\to\mathcal I.$$ We want to turn this into an $\Omega$-decomposition of $v$ while keeping track of the index set. 
To this end, we assume there is an index set $\mathcal J$ and maps $\pi,\pi_i$ that make each of  the following diagrams commute:
$$
\xymatrix{\mathcal J^{\widetilde{\mathcal F}(\Omega)} \supseteq\mathcal D \ar[r]^{\quad\pi} \ar[d]_{\mid_i}& \mathcal I^{\widetilde{\mathcal F}(\Psi)}  \ar[d]^{\mid_i}  \\  \mathcal J^{\widetilde{\mathcal F}(\Omega)_i}\supseteq\mathcal D_i \ar[r]^{\quad\pi_i} & \mathcal I^{\widetilde{\mathcal F}(\Psi)_i}}
$$ 
In addition, we assume:
\begin{itemize}
\item $\pi$ and $\pi_i$ are  defined on subsets $\mathcal D$ and $\mathcal D_i$, respectively. 
\item For $\alpha\in \mathcal J^{\widetilde{\mathcal F}(\Omega)}$ we have: $\alpha\in\mathcal D \Leftrightarrow \alpha_{\mid_i}\in\mathcal D_i$ for all $i\in[n].$
\item $\pi\colon\mathcal D\to \mathcal I^{\widetilde{\mathcal F}(\Psi)}  $ is surjective and all fibers (that is, preimages of single elements) have the same cardinality.
\end{itemize} We will see below that these properties can often be found. 
So, given this, for $\beta\colon \widetilde{\mathcal F}(\Omega)_i\to\mathcal J$ we define $$v_\beta^{[i]}:= \left\{ \begin{array}{cl} w_{\pi_i(\beta)}^{[i]}&\colon  \beta\in\mathcal D_i\\ 0 &\colon \beta\notin\mathcal D_i.\end{array}\right.$$  We now compute
\begin{align*}
\sum_{\alpha\in \mathcal J^{\widetilde{\mathcal F}(\Omega)}} v_{\alpha_{\mid_0}}^{[0]} \otimes \cdots\otimes  v_{\alpha_{\mid_n}}^{[n]}&= \sum_{\alpha\in \mathcal D} w_{\pi_0(\alpha_{\mid_0})}^{[0]}\otimes\cdots \otimes w_{\pi_n(\alpha_{\mid_n})}^{[n]}\\
&\sim \sum_{\alpha\in \mathcal I^{\widetilde{\mathcal F}(\Psi)}} w_{\alpha_{\mid_0}}^{[0]}\otimes\cdots \otimes w_{\alpha_{\mid_n}}^{[n]}=v. 
\end{align*} 
For the first equation we have used that $\alpha\in\mathcal D$ if and only if  all $\alpha_{\mid_i}\in\mathcal D_i$. For the second  equation we have used surjectivity of $\pi$, and that all fibers have the same cardinality. We have thus provided an $\Omega$-decomposition of $v$ over the index set $\mathcal J$. Note that the local vectors from the arising $\Omega$-decomposition are among the ones from the inital $\Psi$-decomposition. So this construction also transforms a separable $\Psi$-decomposition into a separable $\Omega$-decomposition.

(ii) A special case of  the construction in (i) is the following. 
Assume there is  a set $X$ and compatible embeddings $$\xymatrix{ \widetilde{\mathcal F}(\Psi)\ar@{^{(}->}[r]^{\iota}  & X\times \widetilde{\mathcal F}(\Omega)  \\  \widetilde{\mathcal F}(\Psi)_i\ar@{^{(}->}[r]^{\iota_i} \ar@{}[u]|-*[@]{\subseteq}& X\times \widetilde{\mathcal F}(\Omega)_i  \ar@{}[u]|-*[@]{\subseteq}.}$$  For any index set $\mathcal I$ define $$\mathcal J:=\mathcal I^X$$ and obtain   an induced  commutative
  diagram of restriction maps
$$\xymatrix{\mathcal J^{\widetilde{\mathcal F}(\Omega)}=\mathcal I^{X\times\widetilde{\mathcal F}(\Omega)} \ar[r]^{\qquad \pi} \ar[d]_{\mid_i} &\mathcal I^{\widetilde{\mathcal F}(\Psi)} \ar[d]_{\mid_i}  \\   \mathcal J^{\widetilde{\mathcal F}(\Omega)_i}=\mathcal I^{X\times\widetilde{\mathcal F}(\Omega)_i} \ar[r]^{\qquad \pi_i} &\mathcal I^{\widetilde{\mathcal F}(\Psi)_i}  }$$ The above conditions are then obviously fulfilled. Note that $$\vert\mathcal J\vert =\vert\mathcal I\vert^{\vert X\vert}$$ holds in this case.
\demo\end{construction}

\begin{proposition}\label{prop:change} If $\Omega$ is connected, then for every other wsc $\Psi$ on $[n]$  and every  $v\in\mathcal V$ we have $${\rm rank}_\Omega(v) \leq  {\rm rank}_{\Psi}(v)^{\left\vert \widetilde{\mathcal{F}}(\Psi)\right\vert}.$$ The same is true for the separable rank and the purification rank.
\end{proposition}
\begin{proof}
We apply \cref{const:change} (i)  with $\mathcal J:=\mathcal I^{\widetilde{\mathcal F}(\Psi)}$, $\mathcal D$ and $\mathcal D_i$ the sets of constant mappings, and $\pi,\pi_i$ the mappings that send a constant function to its image. 
All  conditions from \cref{const:change} (i) are easily checked to hold. From $$\vert \mathcal J \vert =\vert\mathcal I\vert^{\vert \widetilde{\mathcal F}(\Psi)\vert}$$ the result follows. 
\end{proof}

\begin{example} \label{ex:ineqtr}
For every connected wsc $\Omega$ we have $${\rm rank}_\Omega(v)\leq {\rm rank}_{\Sigma_n}(v).$$ This is clear since $\widetilde{\mathcal F}(\Sigma_n)$ is a singleton. The rank on the right is just the usual tensor rank. This is exactly what was shown in the proof of \cref{thm:nogroup}.
\demo\end{example}

One can improve upon \cref{prop:change} if the two complexes  are not completely independent of each other. We first consider the case $\Psi=m\Omega$ for some $m\in\N.$

\begin{proposition}
For every  wsc $\Omega$,   $m\in\N$ and   $v\in\mathcal V$ we have $${\rm rank}_{m\Omega}(v)\leq \lceil{{\rm rank}_\Omega(v)^{1/m}}\rceil\ \mbox{ and }\ {\rm rank}_{\Omega}(v)\leq {\rm rank}_{m\Omega}(v)^m.$$ The same is true for the separable rank and the purification rank.
\end{proposition}
\begin{proof}
The multiset $\widetilde{\mathcal F}(m\Omega)$ arises from  $\widetilde{\mathcal F}(\Omega)$ by raising the multiplicity of each element by a factor of $m$. So the first inequality is obtained via \cref{const:change} (i)  by choosing an index set $\mathcal J$ with $\mathcal I\hookrightarrow \mathcal{J}^m$ and taking as $\mathcal D,\mathcal D_i$ those functions that only take values in $\mathcal I$.
The second inequality is obtained via   \cref{const:change} (ii) by choosing $X$ as a set with $m$ elements. 
\end{proof}

Finally, for the next result, let $G\neq\{e\}$ be a finite group with two generating sets $S,T$ as in \cref{ex:wsc} (vi), giving rise to the  Cayley complexes $\mathcal C(G,S)$ and $\mathcal C(G,T).$ 

\begin{proposition}   For every $v\in\mathcal V$ we have $${\rm rank}_{\mathcal C(G,S)}(v)\leq  {\rm rank}_{\mathcal C(G,T)}(v)^{\left\lceil{\frac{\vert T\vert}{\vert S\vert}}\right\rceil}.$$The same inequality is  true for the separable rank and the purification rank.
\end{proposition}
\begin{proof} This is obtained through  \cref{const:change} (ii) by choosing a set $X$ such that there is an injective mapping $\varphi=(\varphi_1,\varphi_2)\colon T\hookrightarrow X\times S$. The inclusions from  \cref{const:change} (ii) are then obtained as  \begin{equation}\tag*{} (g,gt) \mapsto (\varphi_1(t),(g, g\varphi_2(t))), \end{equation} where we regard the multiedges as directed edges in the Cayley graph.
\end{proof}

\begin{example}
We consider the group $C_5$ with the two generating sets $S=\{1\}$ and $T=\{1,2\}$. 

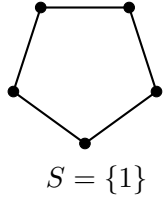
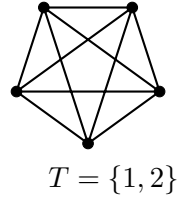
\bigskip
\begin{center}
\begin{tikzpicture}
\filldraw ({cos(0-18)-5},{sin(0-18)}) circle (2pt);
\filldraw ({cos(1*360/5-18)-5},{sin(1*360/5-18)}) circle (2pt);
\filldraw ({cos(2*360/5-18)-5},{sin(2*360/5-18)}) circle (2pt);
\filldraw ({cos(3*360/5-18)-5},{sin(3*360/5-18)}) circle (2pt);
\filldraw ({cos(4*360/5-18)-5},{sin(4*360/5-18)}) circle (2pt);
\draw[thick]  ({cos(0-18)-5},{sin(0-18)}) --  ({cos(1*360/5-18)-5},{sin(1*360/5-18)}) --({cos(2*360/5-18)-5},{sin(2*360/5-18)}) --({cos(3*360/5-18)-5},{sin(3*360/5-18)}) -- ({cos(4*360/5-18)-5},{sin(4*360/5-18)}) --cycle;
\put(-157,-45) {$S=\{1\}$}; 
\filldraw ({cos(0-18)},{sin(0-18)}) circle (2pt);
\filldraw ({cos(1*360/5-18)},{sin(1*360/5-18)}) circle (2pt);
\filldraw ({cos(2*360/5-18)},{sin(2*360/5-18)}) circle (2pt);
\filldraw ({cos(3*360/5-18)},{sin(3*360/5-18)}) circle (2pt);
\filldraw ({cos(4*360/5-18)},{sin(4*360/5-18)}) circle (2pt);
\draw[thick]  ({cos(0-18)},{sin(0-18)}) --  ({cos(1*360/5-18)},{sin(1*360/5-18)}) --({cos(2*360/5-18)},{sin(2*360/5-18)}) --({cos(3*360/5-18)},{sin(3*360/5-18)}) -- ({cos(4*360/5-18)},{sin(4*360/5-18)}) --cycle;
\draw[thick]  ({cos(0-18)},{sin(0-18)}) -- ({cos(2*360/5-18)},{sin(2*360/5-18)})  -- ({cos(4*360/5-18)},{sin(4*360/5-18)}) --  ({cos(1*360/5-18)},{sin(1*360/5-18)}) -- ({cos(3*360/5-18)},{sin(3*360/5-18)}) --cycle;
\put(-15,-45) {$T=\{1,2\}$}; 
\end{tikzpicture}
\end{center}
\vspace{1cm}
We obtain \begin{equation}\tag*{}{\rm rank}_{\mathcal C(C_5,T)}(v)\leq {\rm rank}_{\mathcal C(C_5,S)}(v)\leq  {\rm rank}_{\mathcal C(C_5,T)}(v)^{2}\end{equation} for all $v\in\mathcal V.$ 
\demo\end{example}

\section{Applications to nonnegative tensors}\label{sec:nn}

In this section we consider the global space $$\mathcal V=\C^{d_0}\otimes\cdots \otimes \C^{d_n},$$ and in there the convex cone of tensors with  nonnegative entries. For $n=1$ we have $$\mathcal V={\rm Mat}_{d_0,d_1}(\C)$$ and we thus study nonnegative matrices. For nonnegative matrices, several factorisations and corresponding ranks have been proposed and studied in the literature (see, e.g., \cite{Be94}), 
with applications to  areas such as discrete geometry, 
lifting techniques in optimisation 
and quantum theory (see for example \cite{De19} and references therein). 
 Among them are the rank, the nonnegative rank  \cite{Ya91}, the psd-rank \cite{Fi15,Fa15b}, the completely positive rank \cite{Be15b} and the completely positive semidefinite (transposed) rank \cite{La15c,De19}. 
 We will now generalise these ranks to the multipartite case, and prove a correspondence to the ranks considered in \cref{sec:inv}. This will  allow us to obtain inequalities for the new ranks as easy corollaries of the already proven inequalities.

Any $M\in\mathcal V=\C^{d_0}\otimes\cdots \otimes \C^{d_n}$ can be written as   
$$M:=\sum_{i_0,\ldots, i_n} m_{i_0\ldots i_n} e_{i_0}\otimes \cdots \otimes e_{i_n}$$ 
where $e_{j}$ is the vector with a 1 in position $j$ and 0 else, and where the
complex numbers $m_{i_0\ldots i_n} $ are uniquely defined. 
By definition 
 $M$ is \emph{nonnegative}  if all $m_{i_0\ldots i_n}$ are nonnegative real numbers. To such $M$ we now associate a diagonal matrix $$\sigma\in {\rm Mat}_{d_0}(\C)\otimes \cdots\otimes {\rm Mat}_{d_n}(\C)\cong {\rm Mat}_{d_0\cdots d_n}(\C)$$ by the formula
$$\sigma=\sum_{i_0,\ldots, i_n} m_{i_0\ldots i_n} E_{i_0i_0}\otimes \cdots\otimes E_{i_ni_n},$$
where $E_{jk}$ is the matrix with a 1 in position $(j,k)$ and 0 everywhere else.  Clearly $\sigma$ is positive semidefinite if and only if $M$ is nonnegative, and in this case $\sigma$ is automatically separable. We now define  different types of nonnegative  decompositions  of $M$, 
and relate them to   decompositions of $\sigma$  considered in \cref{sec:inv}. 
This is a generalisation of  the results in \cite[Section 4]{De19}  to the multipartite case on a general wsc. So throughout this section we let $\Omega$ be a connected wsc on $[n]$ equipped with an action from the group $G$. Clearly, $G$-invariance of $\sigma$ is equivalent to $G$-invariance of $M$.

\begin{definition} Let $M\in \C^{d_0}\otimes \cdots \otimes \C^{d_n}$ be given.
\begin{itemize}
\item[(i)] \emph{An \ogd{} of $M$} is defined exactly as in  \cref{def:ogd} (we repeat the definition here to obtain a  consistent numbering). 
\item[(ii)] \emph{A nonnegative \ogd{} of $M$} is an \ogd{} of $M$ in which all local vectors have real nonnegative entries. The corresponding rank is called the \emph{nonnegative $(\Omega,G)$-rank of $M$,} denoted  ${\rm nn\mbox{-}rank}_{(\Omega,G)}(M).$
\item[(iii)] \emph{A positive semidefinite \ogd{} of $M$} consists of positive semidefinite matrices $$0\leqslant E_j^{[i]}\in {\rm Mat}_{\mathcal I^{\widetilde{\mathcal F}_i}}(\C)$$ for $i=0,\ldots, n$ and $j=1,\ldots, d_i,$ such that  $$\left(E_j^{[gi]}\right)_{{}^g\beta, {}^g\beta'}=\left(E_j^{[i]}\right)_{\beta,\beta'}$$ for all $i,g,j,\beta,\beta'$, and $$m_{i_0\ldots i_n}=\sum_{\alpha,\alpha'\in\mathcal I^{\widetilde{\mathcal F}}} \left(E_{i_0}^{[0]}\right)_{\alpha_{\mid_0}, \alpha'_{\mid_0}}\cdots \left(E_{i_n}^{[n]}\right)_{\alpha_{\mid_n}, \alpha'_{\mid_n}}$$ for all $i_0, \ldots, i_n$. 
The smallest cardinality of such an index set $\mathcal I$ is called the  \emph{positive semidefinite $(\Omega,G)$-rank of $M$}, denoted ${\rm psd\mbox{-}rank}_{(\Omega,G)}(M).$\demo
\end{itemize}
\end{definition}

\begin{remark}
(i) In the case $n=1$ and $\Omega=\Lambda_1=\Sigma_1$, i.e.\ the simple edge, the different ranks of $M$ are precisely the matrix  ranks listed in \cite[Section 4.2]{De19}. Without a group action, the $\Lambda_1$-rank is the  rank, the nonnegative $\Lambda_1$-rank is the nonnegative rank, and the positive semidefinite $\Lambda_1$-rank is the psd-rank. With action from $G=C_2$, the $(\Lambda_1,C_2)$-rank is the symmetric rank,  the nonnegative $(\Lambda_1,C_2)$-rank is the cp-rank, and the positive semidefinite  $(\Lambda_1,C_2)$-rank  is the cpsdt-rank.

(ii) For $M$ to have a nonnegative $(\Omega,G)$-decomposition, it is necessary that $M$ is $G$-invariant and that every entry of $M$ is nonnegative. If the action of $G$  on the connected wsc $\Omega$ is free, $G$-invariance and nonnegativity is also sufficient for a nonnegative $(\Omega,G)$-decomposition. This follows directly from \cref{rem:cones}, and also from \cref{thm:corresp} (ii)  in combination with \cref{thm:sep}.

(iii) For $M$ to have a positive semidefinite $(\Omega,G)$-decomposition, it is necessary that $M$ be $G$-invariant and that all entries of $M$ be nonnegative. The upcoming \cref{thm:corresp} (iii), together with \cref{prop:puriex}, ensures sufficiency of these conditions in many cases.\demo\end{remark}

The following is a generalisation of  \cite[Theorem 38]{De19}, which is recovered in the case that $\Omega=\Sigma_1=\Lambda_1$, and $G$ is the trivial group or $C_2$. 

\begin{theorem}\label{thm:corresp}
Let $G$ act on the connected wsc $\Omega$, and let  $M$ and $\sigma$  be defined as above. We then have:
\begin{itemize}
\item[(i)] $\omgr(M)=\omgr(\sigma)$
\item[(ii)] ${\rm nn\mbox{-}rank}_{(\Omega,G)}(M)={\rm sep\mbox{-}rank}_{(\Omega,G)}(\sigma)$
\item[(iii)]${\rm psd\mbox{-}rank}_{(\Omega,G)}(M)={\rm puri\mbox{-}rank}_{(\Omega,G)}(\sigma).$
 \end{itemize}
\end{theorem}
\begin{proof} For (i) start with an $(\Omega,G)$-decomposition $$M=\sum_{\alpha\in\mathcal I^{\widetilde{\mathcal F}}}w_{\alpha_{\mid_0}}^{[0]}\otimes\cdots \otimes w_{\alpha_{\mid_n}}^{[n]}$$ with all $w_\beta^{[i]}\in\C^{d_i}.$ Define $$v_\beta^{[i]}:={\rm diag}(w_\beta^{[i]})\in {\rm Mat}_{d_i}(\C)$$ and immediately check that the $v_\beta^{[i]}$ provide an $(\Omega,G)$-decomposition of $\sigma.$ Conversely, if $$\sigma=\sum_{\alpha\in\mathcal I^{\widetilde{\mathcal F}}} w_{\alpha_{\mid_0}}^{[0]}\otimes \cdots \otimes  w_{\alpha_{\mid_n}}^{[n]}$$ is an \ogd{} of $\sigma,$ the vectors  $v_\beta^{[i]}\in\C^{d_i}$ that contain the diagonal elements of $w_\beta^{[i]}$ provide an \ogd{} of $M$. 

(ii) is proven exactly the same way, using that nonnegative diagonal matrices are psd and that psd matrices have a nonnegative diagonal.

For (iii) let $E_j^{[i]}$ be the psd matrices from a positive semidefinite \ogd{} of $M.$ Write  a Gram decomposition $$E_j^{[i]}=\left(a_{j,\beta}^{[i]^*} a_{j,\beta'}^{[i]}\right)_{\beta,\beta'}$$ with column vectors $a_{j,\beta}^{[i]}\in\C^{\mathcal I^{\widetilde{\mathcal F}_i}}$ such that   $$a_{j,{}^g\beta}^{[gi]}=a_{j,\beta}^{[i]}.$$ This is for example possible by taking the columns of $\sqrt{E_j^{[i]}}$ as the $a_{j,\beta}^{[i]}$.   Now  set $$\tau_\beta^{[i]}:=\sum_{j=1}^{d_i} a_{j,\beta}^{[i]}\otimes E_{jj}$$ and easily check that these local vectors  provide an \ogd{} of a purification $\xi$ of $\sigma$.
Conversely, let $$\xi\in {\rm Mat}_{d_0',d_0}(\C)\otimes\cdots\otimes {\rm Mat}_{d_n',d_n}(\C)$$ be an $(\Omega,G)$-purification of $\sigma$. Denote the local matrices from an \ogd{} of $\xi$ by $\tau_\beta^{[i]}$ and  define $$E_j^{[i]}:=\left(\left( \tau_\beta^{[i]^*}\tau_{\beta'}^{[i]}\right)_{jj}\right)_{\beta,\beta'}.$$ Is is now easily verified that this provides a positive semidefinite $(\Omega,G)$-decompo\-sition of $M$.
\end{proof}

The following corollary generalises several known inequalities on matrix ranks to the multipartite setup.

\begin{corollary} \label{cor:corresp}
Let $G$ act on the connected wsc $\Omega$. Then for all $M\in\C^{d_0}\otimes\cdots\otimes \C^{d_n}$ we have: 
\begin{itemize}
\item[(i)] $ \omgr(M)\leq {\rm nn\mbox{-}rank}_{(\Omega,G)}(M)$
\item[(ii)] ${\rm psd\mbox{-}rank}_{(\Omega,G)}(M) \leq {\rm nn\mbox{-}rank}_{(\Omega,G)}(M)$
 \item[(iii)] $\omgr(M)\leq {\rm psd\mbox{-}rank}_{(\Omega,G)}(M)^2.$
\end{itemize}
\end{corollary}
\begin{proof} This is immediate from \cref{thm:corresp} and \cref{prop:ineq1}.
\end{proof}

\section{Conclusions and Outlook}
\label{sec:outlook}

We have considered an element $v$ in a tensor product space and have  defined and studied an \ogd{} thereof. 
The weighted simplicial complex $\Omega$ determines the arrangement of the indices in the sum of elementary tensor factors, 
and  the decomposition is explicitly invariant under the group action of $G$. 

We have shown that every group action can be refined to become free (\cref{prop:extend}), and 
that if the group action is free then an \ogd{} exists for all invariant $v$ (\cref{thm:ogd}). 
These two results prove the existence of an \ogd{} after possibly increasing the weights of the facets in $\Omega$. 
In addition, we have shown that if the action of $G$ is blending, then every invariant $v$ has an \ogd{} as well (\cref{thm:genex}). 
We have also defined the separable \ogd{} and the $(\Omega,G)$-purification form, 
and have shown that they exist if the action of $G$ is free (\cref{thm:sep}), 
or free or blending (\cref{prop:puriex}), respectively.

We have also provided many inequalities between the different ranks, of which we would like to highlight the following. 
First we have generalised the relations between the $\omgr$ and its separable and purification counterpart  
(\cref{prop:ineq1}),  
then we have shown how much the rank increases when imposing invariance under a certain group action 
(\cref{thm:ineqfree}),
and finally we have shown that the tensor rank is the largest of all ranks 
(\cref{ex:ineqtr}). 

Finally, we  have applied our framework to nonnegative tensors, 
where we have first defined a nonnegative and a positive semidefinite \ogd{},
and then proved a correspondence with the decompositions above (\cref{thm:corresp}),
 and a generalisation of the results of \cite{De19} (\cref{cor:corresp}).

We remark that all  our existence results are constructive, meaning that they can be used to transform the tensor decomposition of \eqref{eq:tensdecomp} into an $(\Omega,G)$-decomposition. Only in the proof of  \cref{thm:genex} we have used a symmetric  decomposition of a specific finite-dimensional tensor, but this can also be found algorithmically, see for example \cite{Br10b}  (and it has to be found only once, independently of the vector in consideration).

An interesting open question is to find  effective algorithms that produce the optimal \ogd, and to analyse their computational complexity. 
On the line, such a procedure is described in \cite[Remark 6]{De19}, 
the symmetric case on the $n$-simplex is covered in \cite{Br10b}, 
and many other cases are analysed in \cite{Hi13}, 
but the general case is open.

A further question concerns approximate decompositions, which are often of great importance. 
The simplest example are low-rank approximations of matrices, which have numerous applications in data science and modelling, but for higher-dimensional tensors the best low-rank approximation might not even  exist \cite{Si08}. 
It would be interesting to study the existence and hardness of approximate \ogd s.

The latter could also be studied from another angle:  
Instead of fixing a rank and looking for the best approximation not exceeding this rank, one could fix a neighborhood of the tensor in consideration, and search for the tensor of lowest rank in this neighborhood. 
This is related to the notions of  \emph{border rank} and  \emph{border tensors}. 
While for matrices the usual rank can only decrease in the limit of a sequence, this fails to be true for higher-order tensors \cite{Bi80}. 
It would be interesting to examine under which conditions on the simplicial complex and the group action this phenomenon appears.


\end{document}